\documentclass{article}

\usepackage{graphicx}
\usepackage[breaklinks=true]{hyperref}

\usepackage{amssymb, amsmath, amsthm}
\newcommand{\eq}{\begin{equation}}
\newcommand{\en}{\end{equation}}
\newcommand{\re}[1]{\mbox{(\ref{#1})}}

\newtheorem{Theorem}{Theorem}
\newtheorem{theorem}[Theorem]{Theorem}
\newtheorem{lemma}[Theorem]{Lemma}
\newtheorem{corollary}[Theorem]{Corollary}
\newtheorem{construction}[Theorem]{Construction}
\newtheorem{proposition}[Theorem]{Proposition}
\newtheorem{example}[Theorem]{Example}
\newtheorem{defn}[Theorem]{Definition}

\newtheorem{question}[Theorem]{Question}
\newtheorem{conjecture}[Theorem]{Conjecture}
\newtheorem{condition}[Theorem]{Condition}
\newtheorem{problem}[Theorem]{Problem}
\theoremstyle{definition}
\newtheorem{remark}[Theorem]{Remark}

\def\Cinlar{\c{C}inlar}
\def\cmaj {concave majorant}
\def\cmin {convex minorant}

\newcommand {\ovC} {\overline{C}}
\newcommand {\unC} {\underline{C}}
\newcommand {\unc} {\underline{c}}

\newcommand {\IR} {\mbox{\msbm\symbol{'122}}}
\newcommand{\IP}{\mathbb{P}}
\newcommand{\IE}{\mathbb{E}}

\newcommand{\ints}{\mathbb{Z}}
\newcommand{\IZ}{\mathbb{Z}}

\newcommand {\IN} {\mbox{\msbm\symbol{'116}}}

\newcommand{\te}{\rightarrow}
\newcommand{\ed}{\mbox{$ \ \stackrel{d}{=}$ }}
\newcommand{\convd}{\mbox{$ \ \stackrel{\!d}{\rightarrow}$ }}

\newcommand{\lb}[1]{\label{#1}}

\newenvironment{prp}[1]{\begin{proposition}\protect\label{#1}}{\end{proposition}}

\newcommand {\ee}{e}
\newcommand {\giv}{ \,|\,}

\newcommand{\ex}{B^{\mbox{$\scriptstyle{\rm ex}$}}}
\newcommand{\me}{B^{\mbox{$\scriptstyle{\rm me}$}}}

\newcommand{\hf}{ \mbox{${1 \over 2}$}}

\newcommand{\thf}{ \mbox{${3 \over 2}$}}
\newcommand{\fhf}{ \mbox{${5 \over 2}$}}

\newcommand{\hR}{\widehat{R}}

\newcommand {\arcsinh} {{\rm arcsinh}}
\newcommand {\erfc} {{\rm erfc}}
\newcommand{\erf}{{\rm erf}}
\def\IR{\mathbb{R}}
\def\IN{\mathbb{N}}
\def\barU{\overline{U}}

\def\ba#1\ee{\begin{align*}#1\end{align*}}
\def\ban#1\ee{\begin{align}#1\end{align}}

\newcommand {\arcosh} {{\rm{arcosh}}}
\newcommand {\arcos} {{\rm arcos}}
\newcommand {\argmin} {{\rm argmin}}
\newcommand {\htau} {    {\hat{\tau}}  }
\newcommand{\ar}{{\rm a }}

\newcommand{\E}{\IE}

\newcommand{\tr}{$(\tau, \rho)$}
\newtheorem{definition}[Theorem]{Definition}

\newcommand {\alphabr} {\alpha^\circ}
\newcommand {\taubr} {\tau^\circ}
\newcommand {\htaubr} {    {\hat{\tau}^\circ}  }
\newcommand {\rhobr} {\rho^\circ}
\newcommand{\gh}{\hat{g}}
\newcommand{\hp}{\left(P\hat{g}\right)}

\begin{document}

\title{The greatest convex minorant of Brownian motion, meander, and bridge}

\author{Jim Pitman\thanks{University of California at Berkeley; email pitman@stat.Berkeley.EDU; research supported in part by N.S.F.\ Grant DMS-0806118} 
\,\,\,\,\,\,Nathan Ross\thanks{University of California at Berkeley; email ross@stat.Berkeley.EDU}}

\maketitle

\begin{abstract} 
This article contains both a point process and a sequential description of the greatest \cmin\ of Brownian motion on a finite interval.
We use these descriptions to provide new analysis of various features of the \cmin\
such as the set of times where the Brownian motion meets its minorant. 
The equivalence of the these descriptions is non-trivial, which leads to many interesting identities between quantities
derived from our analysis. The sequential description can be viewed as a Markov chain for
which we derive some fundamental properties.

%\em AMS 2000 subject classifications: \em \\
%\em Keywords: \em continuum random tree,  
\end{abstract}

\section{Introduction}

The {\em greatest convex minorant} (or simply \cmin\ for short) of a real-valued function $(x_u, u \in U)$ with domain $U$ contained in the real line
is the maximal convex function $(\unc_u, u \in I)$ defined on a closed interval $I$ containing $U$ with $\unc_u \le x_u$ for all $u \in U$.
A number of authors have provided descriptions of certain features of the \cmin\ for various stochastic processes such as random walks \cite{MR994088},
Brownian motion \cite{MR2007793, brownistan, gboom83, p83, suidan01}, Cauchy processes \cite{MR1747095}, Markov Processes \cite{MR770946}, and L\'evy processes (Chapter XI of \cite{MR1739699}).

In this article, we will give two descriptions of the \cmin\ of various Brownian path fragments which yield new insight into the structure of the \cmin\
of a Brownian motion over a finite interval.
As we shall see below, such a \cmin\ is a piecewise linear function with infinitely many linear segments which accumulate
only at the endpoints of the interval.  We refer to linear segments as ``faces,"
the ``length" of a face is as projected onto the horizontal time axis, and the slope of a face is the slope
of the corresponding segment.  We also refer to the points where the \cmin\ equals the process as vertices; note that these points are
also the endpoints of the linear segments.  See figure \ref{11} for illustration.

\begin{figure}[h]
\begin{center}
\includegraphics[scale=.75]{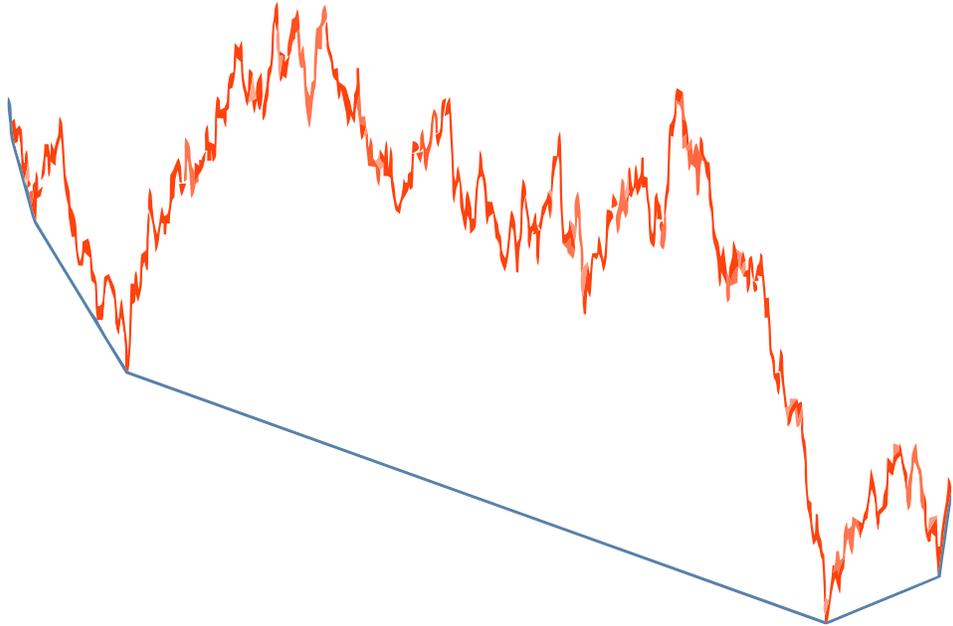}
\end{center}
\caption{A typical instance of a finite time Brownian motion and its \cmin.  The ``faces" of the \cmin\ are the linear segments, the ``lengths" are
as projected to a horizontal axis, and the ``slope" is the slope of the segment.}\label{11}
\end{figure}

Our first description is a Poisson point process of the lengths and slopes of the faces
of the \cmin\ of Brownian motion on an interval of a random exponential length.
This result can be derived from the recent developments of \cite{ap10} and \cite{pub10} 
and is in the spirit of previous studies of the \cmin\ of Brownian motion run to infinity (e.g. \cite{gboom83}).
We provide a proof below in Section~\ref{pppd}.

\begin{theorem}\label{bmpp}
Let $\Gamma_1$ an exponential random variable with rate one.
The lengths~$x$ and slopes~$s$ of the faces of the \cmin\ of a Brownian motion on 
$[0, \Gamma_1]$
form a Poisson point process on $\IR^+ \times \IR$ with intensity measure
\begin{align}
\frac{\exp\{-\frac{x}{2}\left(2+s^{2}\right)\}}  {\sqrt{2 \pi x}} \, ds\,dx, \hspace{5mm} x\geq0, s\in \IR. \label{imbm}
\end{align}
\end{theorem}

We will pay special attention to the set of times of the vertices of the \cmin\ of a Brownian motion on 
$[0,1]$.  To this end, let
\ban
0 < \cdots < \alpha_{-2} < \alpha_{-1} < \alpha_0 < \alpha_1 < \alpha_2 < \cdots < 1 \label{alps}
\ee
with $\alpha_{-n} \downarrow 0$ and $\alpha_n \uparrow 1$ as $n \te \infty$ denote the times
of vertices of the convex minorant of a Brownian motion $B$ on $[0,1]$, arranged relative to
\ban
\alpha_0 := \argmin_{ 0 \le t \le 1 } B_t. \label{alp0}
\ee
Theorem \ref{bmpp} implicitly contains the distribution of the sequence $(\alpha_i)_{i\in\IZ}$.  This description is
precisely stated in the following corollary of Theorem \ref{bmpp}, which follows easily from Brownian scaling.
\begin{corollary}\label{lit}
If $\{(L_i, S_i), i \in \IZ\}$ are the lengths and slopes given by the Poisson point process with intensity measure \eqref{imbm},
arranged so that 
\ba
\cdots S_{-1}<S_{0}<0<S_1<S_2\cdots
\ee
then
\ba
\left(\alpha_n\right)_{n\in\IZ}\ed\left(\sum_{i\leq n}L_i\bigg/\sum_{i\in\IZ}L_i\right)_{n\in\IZ}.
\ee
%where $\Gamma_1$ is independent of $(\alpha_n)_{n\in\IZ}$.
\end{corollary}
%\begin{proof}
%The proposition follows from Theorem \ref{bmpp} and the Brownian scaling
%\ba
%\left(\sqrt{\Gamma_1}B(t/\Gamma_1),0\leq t\leq \Gamma_1\right)\ed\left(B(t),  0\leq t\leq \Gamma_1\right).
%\ee
%\end{proof} 

Our second description provides a Markovian recursion for the vertices of the \cmin\ of a Brownian meander (and Bessel(3) process and bridge), which applies to
Brownian motion on a finite interval through Denisov's decomposition at the minimum \cite{MR726906} - background on these concepts is provided in
Section \ref{background}.
In our setting, Denisov's decomposition of Brownian motion on $[0,1]$ states that conditional on $\alpha_0$,
the pre and post minimum processes are independent Brownian meanders of appropriate lengths.  
We now make the following definition.
\begin{definition}
We say that a sequence of random variables $(\tau_n, \rho_n)_{n\geq0}$ satisfies the \tr\ recursion if
for all $n \ge 0$:
\ba
\rho_{n+1} = U_n \rho_{n} 
\ee
and
\ba
\tau_{n+1} = \frac{ \tau_n  \rho_{n+1}^2 }{ \tau_n Z_{n+1}^2 + \rho_{n+1}^2 }
\ee
for the two independent sequences of i.i.d. uniform $(0,1)$ variables $U_n$ and i.i.d. squares of standard normal random variables $Z_n^2$, both independent of $(\tau_0, \rho_0)$.
\end{definition}

\begin{figure}[h]
\begin{center}
\includegraphics[scale=.75]{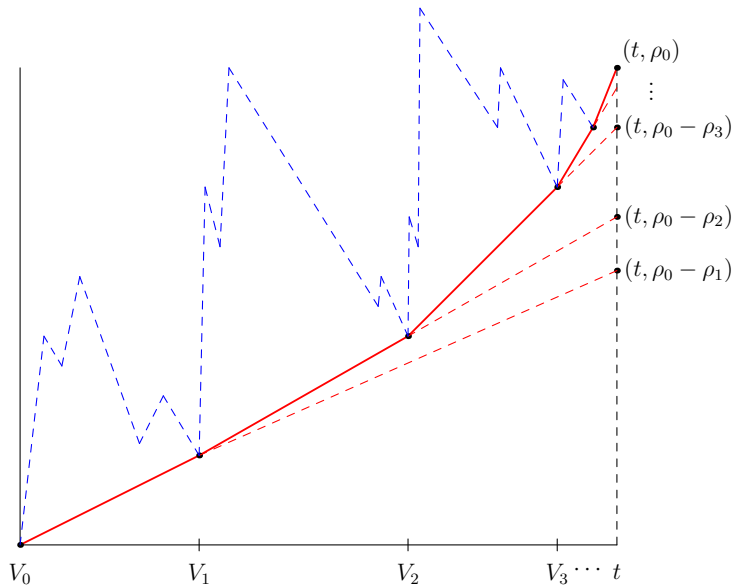}
\end{center}
\caption{An illustration of the notation of Theorem \ref{mean}.  The dashed line represents a Brownian meander of length $t$,
and the solid line its \cmin.  Note also that $V_i:= t-\tau_i$ for $i=0, 1, \ldots$}\label{22}
\end{figure}

\begin{theorem}\label{mean}
Let $(X(v), 0 \le v \le t)$ be a Brownian meander of length $t$, and let
$(\unC(v), 0 \le v \le t)$ be its \cmin. The vertices of $(\unC(v), 0 \le v \le t)$ occur at times
$0 = V_0 < V_1 < V_2  < \cdots$ with $\lim_n V_n = t$. Let $\tau_n:= t - V_n$ so $\tau_0 = t > \tau_1 > \tau_2 > \cdots$
with $\lim_n  \tau_n  = 0$. Let $\rho_0 = X(t)$ and for $n\geq1$ let $\rho_0- \rho_n$ denote the intercept at time $t$ of the line extending the segment of the
\cmin\ of $X$ on the interval $(V_{n-1}, V_n)$. The \cmin\ $\unC$ of $X$ is uniquely determined by the sequence
of pairs $(\tau_n, \rho_n)$ for $n = 1, 2, \ldots$ which satisfies the \tr\ recursion with
\eq
\rho_0 \ed \sqrt{2 t \Gamma_1 } \mbox{ and } \tau_0 = t,
\en
where $\Gamma_1$ is an exponential random variable with rate one.
\end{theorem}

Once again, Theorem \ref{mean} implicitly contains the distribution of the sequence $(\alpha_i)_{i\in\IZ}$, as described in the following corollary
which follows from Denisov's decomposition and Brownian scaling.
\begin{corollary}\label{lighted}
Let $0 = 1 - \tau_0 < 1 - \tau_1  < \cdots$ and $0 = 1 - \htau_0 < 1 - \htau_1  < \cdots$ be the times of the vertices of the \cmin s of
two independent and identically distributed standard Brownian meanders.  Then the sequence $(\alpha_i)_{i\in\IZ}$ of times of vertices 
of the \cmin\ of Brownian motion on $[0,1]$  may be
represented for $n\geq0$ as 
\ba
\alpha_{-n} & = \tau_{n} \alpha_0, \, \\
\alpha_{n} & = 1 - \htau_n ( 1 - \alpha_0 ) \ed 1 - \alpha_{-n}, 
\ee
where $\alpha_0$ is independent of the sequences $(\tau_i)_{i\geq0}$ and $(\htau_i)_{i\geq0}$.
\end{corollary}
%\begin{proof}
%The proposition follows from Denisov's decomposition and Brownian scaling.
%\end{proof}

Corollaries \ref{lit} and \ref{lighted} provide a bridge
between the two descriptions of Theorems \ref{bmpp} and \ref{mean} so that each of these descriptions is implied by the other.
More precisely, we have the following (Brownian free) formulation, where here and below for $s>0$, $\Gamma_s$ denotes a gamma random variable with density
\ba
\frac{x^{s-1}e^{-x}}{\Gamma(s)}, \hspace{5mm} x>0,
\ee
and where $\Gamma(s)$ denotes the gamma function. 
\begin{theorem}\label{equiv}
If the sequence of random variables $(\tau_n, \rho_n)_{n\geq0}$ satisfies the \tr\ recursion with 
\ban
\tau_0\ed\Gamma_{1/2}, \hspace{6mm} \rho_0\ed\sqrt{2\Gamma_{1/2}\Gamma_1}, \label{inici}
\ee
where $\Gamma_{1/2}$ and $\Gamma_1$ are independent, then the random set of pairs
\ba
\left\{ \left(\tau_{i-1}-\tau_i,\sum_{j=1}^i\frac{\rho_{j-1}-\rho_{j}}{\tau_{j-1}} \right): i\in\IN \right\}
\ee
forms a Poisson point process on $\IR^+\times\IR^+$ with intensity measure
\begin{align}
\frac{\exp\{-\frac{x}{2}\left(2+s^{2}\right)\}}  {\sqrt{2 \pi x}} \, ds\,dx, \hspace{5mm} x,s \geq0. \label{imbm2}
\end{align}

Conversely, if $\{(L_i, S_i):i\in\IN\}$ is the set of points of a Poisson point process with intensity measure
given by \eqref{imbm2}, ordered so that $S_0:=0<S_1<S_2<\ldots$ then the variables
\ba
\tau_i=\sum_{j=i+1}^\infty L_j \mbox{\,\, \rm{and} \,\,} \rho_i=\sum_{j=i+1}^\infty S_jL_j - S_i\sum_{j=i+1}^\infty L_j,\,\, i=0,1,2, \ldots
\ee
satisfy the \tr\ recursion with $(\tau_0, \rho_0)$ distributed as in \eqref{inici}.
\end{theorem}
\begin{proof}
The theorem is proved by Brownian scaling in the relevant facts above coupled with the fundamental identity
$\alpha_0\Gamma_1\ed\Gamma_{1/2}$, where $\alpha_0$ and $\Gamma_{1}$ are independent.
\end{proof}
It is not at all obvious how to show Theorem \ref{equiv} directly.  Moreover, many simple 
quantities can be computed and related to both descriptions which we cannot independently show to be equivalent.  
For example, we have the following result which follows from Theorem \ref{equiv}, but for which we do not have an independent proof - see Section \ref{cqn} below.
\begin{corollary}\label{bewid}
Let $W$ and $Z$ standard normal random variables, $U$ uniform on $(0,1)$, and $R$ Rayleigh distributed having density $re^{-r^2/2}$, $r>0$.  If
all of these variables are independent, then 
\ba
\frac{ W^2 + (1-U) ^2 R^2 }{ 1 + U^2 R^2/Z^2 }  \ed Z^2.
\ee
\end{corollary}

The layout of the paper is as follows.  Section \ref{background} contains the notation and much of the background used in the paper.
%the remainder of the introduction contains a discussion of the natural ordering of the interval partition generated by the times of the vertices of the \cmin\ of both a Brownian motion and a Cauchy process.
Sections \ref{pppd} and \ref{seq} respectively contain the Poisson and sequential descriptions of the \cmin\ of various Brownian paths and in Section \ref{cqn} we discuss
identities derived by relating the two descriptions.
In Section \ref{ppc} we derive various densities and transforms associated to the process of vertices and slopes of faces of the \cmin\
and in Section \ref{seqcon} we discuss some aspects (including a CLT) of the Markov process implicit in the sequential construction.

\section{Background}\label{background}

This section recalls some background and terminology for handling various Brownian path fragments.

Let $(B(t), t \ge 0 )$ denote a standard one-dimensional Brownian motion, abbreviated $BM^0$, and let $(R_3(t), t \ge 0 )$ denote a standard 3-dimensional Bessel process, abbreviated $BES^0(3)$,
defined as the square root of the sum of squares of $3$ independent copies of $B$. So $B(0) = R_3(0) = 0$, $E(B(t)^2) = t$ and $E(R_3(t)^2) = 3t$.
The notation $BM^x$ and $BES^x(3)$ will be used to denote these processes with a general initial value $x$ instead of $x = 0$, where necessarily $x \ge 0$ for $BES^x(3)$. 
%Fundamental to intuition is the idea that $R_3$ may be interpreted as $B$ conditioned on the event $(B(t) > 0 \mbox{ for all } t \ge 0 )$. This
%idea has been made rigorous in various ways, e.g. by various approximations, and Doob's theory of $h$-processes \cite{p75}.

\paragraph{Bridges}
For $0 \le s < t $ and real numbers $x$ and $y$, a {\em Brownian bridge from $(s,x)$ to $(t,y)$} is a process identical in law to $(B(u), s \le u \le t )$ given
$B(s) = x$ and $B(t) = y$, constructed to be weakly continuous in $x$ and $y$ for fixed $s$ and $t$. The explicit construction of all such bridges by suitable scaling
of the {\em standard} Brownian bridge from $(0,0)$ to $(1,0)$ is well known, as is the fact that for $B$ a $BM^0$ the process
$$
(B(t) - t B(1), 0 \le t \le 1)
$$
is a standard Brownian bridge independent of $B(1)$.

The family of {\em $BES(3)$ bridges from $(s,x)$ to $(t,y)$} is defined similarly for $0 \le s < t $ and $x,y \ge 0$. 
The $BES(3)$ bridge from $(s,x)$ to $(t,y)$ is a Brownian bridge from $(s,x)$ to $(t,y)$ conditioned to remain strictly positive on $(s,t)$.
For $x >0$ and $y >0$ the conditioning event for the Brownian bridge has a strictly positive probability, so the conditioning is elementary, and 
the assertion is easily verified. If either $x=0$ or $y=0$ the conditioning event has zero probability, and the assertion can either be interpreted
in terms of weak limits as either $x$ or $y$ or both approach $0$, or in terms of $h$-processes \cite{blum_exc83,chaumont-bravo09,fpy92}.

\paragraph{Excursions and meanders}

The $BES(3)$ bridge from $(0,0)$ to $(t,0)$ is known as a {\em Brownian excursion of length $t$}. This process can be constructed by Brownian scaling as
$(\sqrt{t} \ex( v/t) , 0 \le v \le t )$ where  $(\ex(u), 0 \le u \le 1)$ is the {\em standard Brownian excursion} of length $1$.
Intuitively, the Brownian excursion of length $t$ should be understood as $(B(v), 0 \le v \le t)$ conditioned on $B(0) = B(t) = 0$ and $B(v) > 0$ for all $0 < v < t$.
Similarly, conditioning $(B(v), 0 \le v \le t)$ on $B(0) = 0$ and $B(v) > 0$ for all $0 < v < t$, without specifying a value for $B(t)$, leads to the concept of a
{\em Brownian meander of length $t$}. 
This process can be constructed as $(\sqrt{t} \me( v/t) , 0 \le v \le t )$ where  $(\me(u), 0 \le u \le 1)$ is the {\em standard Brownian meander} of length $1$ which for our purposes is best considered via the following result of Imhof \cite{imhof84}.  
\begin{proposition}\cite{imhof84}\label{imh}
If $(R_3(t), 0\leq t \leq 1)$ is a $BES^0(3)$ process, then the process $(\me(u), 0 \le u \le 1)$ is absolutely continuous with respect to the law of
$(R_3(t), 0 \le t \le 1)$,
with density $(\pi/2)^{1/2}x^{-1}$, where $x=R_3(1)$ is the final value of $R_3$ .  Thus, 
$R_3$ and $\me$ share the same collection of $BES(3)$ bridges from $(0,0)$ to $(1,r)$ obtained 
by conditioning on the final value $r$.
\end{proposition}
We also say $(X(t), 0\leq t \leq T)$ is a Brownian meander of \emph{random} length $T>0$,
if $(T^{-1/2}X(uT), 0\leq u\leq 1)\ed(\me(u), 0 \leq u \leq 1)$, with $\me$ independent of $T$.
Informally, $X$ is a random path of random length. Formally, we may represent $X$ as
a random element of $C[0,\infty)$ by stopping the path at time $T$.

We recall the following basic path decomposition for standard Brownian motion run for a finite time due to Denisov \cite{MR726906}.
Recall that the Rayleigh distribution has density $re^{-r^2/2}$ for $r>0$, and the arcsine distribution has density $1/(\pi \sqrt{x (1-x)})$ on $[0,1]$.
\begin{prp}{dendecomp}\cite{MR726906}(Denisov's Decomposition).
Let $(B(u), u \geq0)$ be a Brownian motion, and let $T$ be the a.s. unique time that $B$ attains its minimum on $[0,1]$ and $M=B(T)$ its minimum. 
\begin{itemize}
\item $(T,M)\ed (\beta, -\sqrt{\beta}R)$, where $\beta$ has the arcsine distribution, $R$ has the Rayleigh distribution, and $\beta$ and $R$
are independent.
\item Given $T$, the processes $(B(T-u)-M, 0\leq u \leq T)$ and  $(B(T+u)-M, 0\leq u \leq 1-T)$ are independent Brownian meanders
of lengths $T$ and $1-T$, respectively.
%\item $(B(T-u)-M, 0\leq u \leq T)\ed(\sqrt{T}\me(u/T), 0\leq u \leq T)$ is a Brownian meander of length $T$.
\end{itemize}
\end{prp}
We will frequently use variations of this result derived by Brownian scaling and conditioning; for example we have the following proposition,
which can be viewed as a formulation of Williams decomposition \cite{MR0350881}.
\begin{proposition}\label{denvar}
Let $(B(u), u \geq0)$ be a Brownian motion and $\Gamma_1$ an exponential random variable with rate one
independent of $B$.  Let $T$ be the a.s. unique time 
that $B$ obtains its minimum on $[0, \Gamma_1]$, and $M=B(T)$ its minimum. 
\begin{itemize}
\item $(T,M)\ed (\Gamma_{1/2}, -\sqrt{\Gamma_{1/2}}R)$, where $2\Gamma_{1/2}$ is distributed as the square of a standard normal random variable, and $\Gamma_{1/2}$ and $R$ are independent.
\item The processes $(B(T-u)-M, 0\leq u \leq T)$ and  $(B(T+u)-M, 0\leq u \leq \Gamma_1-T)$ are independent Brownian meanders of 
lengths $T$ and $\Gamma_1-T$, respectively.
\end{itemize}
\end{proposition} 
\begin{proof}
The first item follows by Brownian scaling and the elementary fact that for $\beta$ having the arcsine distribution and $\Gamma_1$ independent of $\beta$, $\Gamma_1\beta \ed \Gamma_{1/2}$.  The second
item is a restatement of the second item of Proposition \ref{dendecomp} after scaling the meanders appropriately. 
\end{proof}

We also have the following basic path decomposition for
$BES(3)$ due to Williams \cite{MR0350881}, which our results heavily exploit. See \cite{MR896736,gp80,MR942038,p75} for various
proofs.

\begin{prp}{willdecomp}\cite{MR0350881}
(Williams decomposition of $BES(3)$).
Let $R_3^r(u), u\ge 0 $ be a $BES^r(3)$ process, and $T$ the time that
$R_3^r$ attains its ultimate minimum. Then
\begin{itemize}
\item $R_3^r(T)$ has uniform distribution on $[0,r]$;
\item  given $R_3^r(T) = a$ the process $(R_3^r(u), 0 \le u \le T)$ is distributed as $(B(u), 0 \le u \le T_a)$ where $B$ is a $BM^r$ and $T_a$ is the first hitting time of $a$ by $B$.
\item  given $R_3^r(T) = a$ and $T=t$ the processes 
$(R_3^r(t-u)-a, 0 \le u \le t)$ and $(R_3^r(t+u)-a, 0 \le u < \infty)$ are
independent, with first a $BES(3)$ bridge from $(0,0)$ to $(t,r-a)$,
and the second a $BES^0(3)$ process.
\end{itemize}
\end{prp}
The third item of Proposition \ref{willdecomp} can be slightly altered by replacing the $BES(3)$ bridge by a Brownian first passage bridge as the proposition below indicates; see \cite{bcp03}.
\begin{proposition}\label{fpbes}
Let $(B(u), u\geq0)$ a standard Brownian motion and for fixed $a>0$, let $T_a=\inf\{t>0:B(t)=a\}$.  Then given $T_a=t$, the process $(a-B(T_a-u), 0\leq u \leq t)$ is equal in distribution
to a $BES(3)$ bridge from $(0,0)$ to $(t,a)$.
\end{proposition}

%\paragraph{Time inversion}
%
%The well known invariance of the distribution $BM^0$ under the time inversion operation
%$$
%\hB(t) := t B(1/t) , ~~~ t > 0 ~~~( \mbox{with } \hB(0):= 0 )
%$$
%plays an important role in our analysis of convex minorants of
%various Brownian path fragments. Computations similar to those below, which we include mainly to ease exposition later, will be used crucially in the sequel.
%
%The Brownian path $(B(u), 0 \le u \le 1)$ can be represented
%as 
%\eq
%\lb{inv1}
%B(u) = u \hB(1/u)    \mbox{ where }  0 < u \le  1 \le 1/u,
%\en
%so that defining $G_1 = \sup \{v \le 1 : B(v) = 0 \}$, 
%we have
%\eq
%\lb{inv2}
%(B(u), 0 \le u \le G_1) = ( u \hB(1/u), 0 \le u \le G_1).
%\en
%It follows that $G_1 = 1/H_1$ where $H_1 = \inf \{t \ge 1: \hB(t) = 0 \}.$
%By the Markov property  of $B$ at time $1$, and Brownian scaling,
%$$
%H_1 \ed 1 + \hB(1)^2 T_1
%$$
%where $T_1$ is the first passage time to one of a standard Brownian motion, independent of $B$.  It is also well known that
%$T_1 \ed 1/Z^2$, where $Z$ is a standard normal random variable.
%Hence we have the identity in law
%\eq
%\lb{inv3}
%G_1 \ed \frac{ 1 } { 1 + \hB(1)^2 /Z^2 } = \frac{ Z^2  } { Z^2  + \hB(1)^2 }
%\en
%where $Z$ and $\hB(1)$ are independent.
%

\section{Poisson point process description}\label{pppd}

In this section we first prove Theorem \ref{bmpp} and then collect some facts about the Poisson point process description contained there.
\begin{proof}[Proof of Theorem \ref{bmpp}]
Let $(\unC(t), 0 \leq t \leq \Gamma_1)$ be the convex minorant of a Brownian motion on $[0, \Gamma_1]$ and let $\unC'(t)$ denote 
the right derivative of $\unC$ at $t$.  
Let $\tau_a=\inf\{t>0:\unC'(t)>a\}$, and note that outside of values of slope of the \cmin\, we can alternatively define $\tau_a=\argmin\{B(t)-at:t>0\}$.  
Now, $(\tau_a, a\in\IR)$ contains all the information about the \cmin\ we need since 
the set
\ba
\{(a,\tau_a-\tau_{a-}): \tau_a-\tau_{a-}>0\}
\ee
correspond 
to slopes and lengths of the \cmin.  

In order to prove the theorem, we basically
need to show that the process $\tau_a$ is an increasing pure jump process with independent increments
with the appropriate Laplace transform. 
Due to the description of $\tau_a$ as the time of the minimum of
Brownian motion with drift on $[0,\Gamma_1]$,
the assertion of pure jumps follows from uniqueness of the minimum of Brownian motion with drift, and the independent increments from the independence of the pre and post minimum processes - see \cite{gp80}
(a more detailed argument of these assertions can be found in  \cite{pub10}).

From this point we only need to show that
the Laplace transform of $\tau_a$ is equal to the corresponding quantity of the ``master equation" of the Poisson point process with intensity measure given by \eqref{imbm} (as this is characterizing in our setting).  
Precisely, we need to show
\begin{align}
\IE e^{-t\tau_a}&=\exp\left\{-\int_0^\infty \left(1-e^{-tx}\right)\int_{-\infty}^a\frac{\exp\{-\frac{x}{2}\left(2+s^{2}\right)\}}  {\sqrt{2 \pi x}}ds dx \right\}.\label{Lapta}
\end{align}
From \cite{gp80} (or \cite{MR1406564} Chapter VI, Theorem 5), we have that 
\begin{align}
\IE e^{-t\tau_a}&=\exp\left\{-\int_0^\infty \left(1-e^{-tx}\right)e^{-x}x^{-1}\IP(B_x-ax<0) dx \right\},\notag 
\end{align}
which is \eqref{Lapta}.
\end{proof}

The next set of results can easily be read from the intensity measure \eqref{imbm}.
\begin{proposition}\label{mnn}
\mbox{}
\begin{enumerate}
\item The slopes of the faces of the \cmin\ of a Brownian motion on $[0,\Gamma_1]$ are given by a Poisson point process with intensity measure
\ba
\int_{0}^\infty \frac{\exp\{-\frac{x}{2}\left(2+s^{2}\right)\}}  {\sqrt{2 \pi x}} dx \, ds = \frac{1}{\sqrt{2+s^2}} ds, \hspace{5mm} s\in\IR.
\ee
\item The lengths of the faces of the \cmin\ of a Brownian motion on $[0,\Gamma_1]$ are given by a Poisson point process with intensity measure
\begin{align}
\int_{-\infty}^\infty \frac{\exp\{-\frac{x}{2}\left(2+s^{2}\right)\}}  {\sqrt{2 \pi x}} ds \,dx = \frac{e^{-x}}{x} dx, \hspace{5mm} x>0. \label{1in}
\end{align}
\item The mean number of faces of the \cmin\ of a Brownian motion on $[0,\Gamma_1]$ having slope in the interval $[a,b]$ is 
\ba
\int_{a}^b\int_{0}^\infty \frac{\exp\{-\frac{x}{2}\left(2+s^{2}\right)\}}  {\sqrt{2 \pi x}} dx ds
&= \log\left(\frac{b+\sqrt{2+b^2}}{a+\sqrt{2+a^2}}\right).
\ee
\item The intensity measure of the Poisson point process of lengths $x$ and increments $y$ of the
\cmin\ of a Brownian motion on $[0,\Gamma_1]$ can be obtained by making the change of 
variable $s=y/x$ in the intensity measure \eqref{imbm} which yields
\begin{align}
\frac{\exp\{-\frac{x}{2}\left(2+(y/x)^{2}\right)\}}  {\sqrt{2 \pi x^3}} \,dx \,dy \hspace{5mm} x>0, y\in\IR. \label{imbi}
\end{align}
\end{enumerate}
\end{proposition}

From this point, we can prove the following result, which can be read from \cite{gboom83}, see also \cite{jpfb09}.
\begin{proposition}\cite{gboom83}
The sequence of times of vertices of the \cmin\ of a Brownian motion on $[0,1]$, denoted $(\alpha_i)_{i\in\IZ}$, has accumulation points only 
at $0$ and $1$.
\end{proposition}
\begin{proof}
The faces of a \cmin\ are arranged in order of increasing slope, and Item 3 of Proposition \ref{mnn} implies
the mean number of faces of the \cmin\ of a Brownian motion on $[0,\Gamma_1]$ with slope in a given interval is finite.
Also note that that 
\ba
\int_{-\infty}^0\int_{0}^\infty \frac{\exp\{-\frac{x}{2}\left(2+s^{2}\right)\}}  {\sqrt{2 \pi x}} dx ds=\infty,
\ee
and hence that the sequence $(\Gamma_1\alpha_i)_{i\in\IZ}$ has
accumulation points only at zero and at $\Gamma_1$ (by symmetry in the integrand).  This last statement implies the result
for the sequence $(\alpha_i)_{i\in\IZ}$.
\end{proof}

Theorem \ref{bmpp} also provides a constructive description of the \cmin\ of Brownian motion on  $[0,\Gamma_1]$.  
%After making the change of 
%variable $s=y/x$ in the intensity measure \eqref{imbm}, we obtain the following proposition.
%\begin{proposition}
%The intensity measure of the Poisson point process of lengths $x$ and increments $y$ of the
%\cmin\ of a Brownian motion on $[0,\Gamma_1]$ is  
%\begin{align}
%\frac{\exp\{-\frac{x}{2}\left(2+(y/x)^{2}\right)\}}  {\sqrt{2 \pi x^3}} \,dx \,dy \hspace{5mm} x,y>0. \label{imbi}
%\end{align}
%\end{proposition}
%
%We can interpret the intensity \eqref{imbi} by
%first choosing the lengths $x$ according to a one dimensional Poisson point process with intensity measure
%given by \eqref{1in}
%and for each length $x$, the increment is normally distributed with mean zero and variance $x$.  The equivalence of this description to that of
%Theorem \ref{bmpp} follows by comparing L\'evy measures.
%The natural ordering of the faces of the \cmin\ is obtained by
%arranging these lengths in order of increasing slope and an additional step provides an It\^o-type
%excursion theory for the path of the Brownian motion above the \cmin, but this discussion would be
%beyond the scope of this paper, see \cite{pub10}.
\begin{theorem}
For $i\geq1$, let $W_i$ independent uniform $[0,1]$ variables and define  
\eq
\lb{stickb}
J_1:=  W_1, ~~J_2:= (1-W_1) W_2, ~~J_3:= (1-W_1) (1 - W_2) W_3, \ldots 
\en
If $B_1, B_2, \ldots$ are independent standard Brownian motions, then
the lengths and increments of the faces of the \cmin\ have the same distribution as
the points $(J_i,B_i(J_i))$.  The distribution of these points determine the distribution of the \cmin\
by reordering the lengths and
increment points with respect to increasing slope.  
\end{theorem}
\begin{proof}
By comparing L\'evy measures, it is not difficult to see that the lengths and increments of the \cmin\ of $B$ on $[0, \Gamma_1]$
can be represented as $(L_i, \sqrt{L_i} Z_i)_{i\in\IZ}$, where $Z_i$ are independent standard normal random variables, and the $L_i$ are the
points of a Poisson point process with intensity given by \eqref{1in}. 
Thus, Brownian scaling implies the \cmin\ of a Brownian motion on $[0, 1]$ has lengths and increments given by 
\ba
(L_i^*, \sqrt{L_i^*} Z_i)_{i\in\IZ}, {\rm where} \,\, L_i^*=L_i/\sum_{j\in\IZ}L_j.
\ee
From this point, the result will follow if we show the following equality in distribution of point processes:
\ban
\{L_i^*\}_{i\in\IZ}\ed\{J_i\}_{i\in\IN}. \label{prt}
\ee

Following Chapter 4 of \cite{MR2245368},
for a $\Gamma_1$ random variable independent of the $J_i$, 
$\Gamma_1J_i$ are the points of a Poisson point process with intensity measure given by \eqref{1in}, so that
\ba
\{L_i\}_{i\in\IZ}\ed\{\Gamma_1J_i\}_{i\in\IN}.
\ee
Since the set $\{J_i\}_{i\in\IN}$ has sum equal to $1$ almost surely \cite{MR2245368},
$J_i = \Gamma_1J_i/\sum_{k\in\IN}\Gamma_1J_k$ almost surely so that \eqref{prt} now follows from the definition of $L_i^*$.
\end{proof}
\begin{remark}
The distribution of the ranked (decreasing) rearrangement of $\{J_i\}_{i\in\IN}$
is known as the Poisson-Dirichlet$(0, 1)$  distribution.  See \cite{MR2245368} for background.
\end{remark}

%
%Note that $\{L_i^*\}_{i\in\IZ}$ is an interval partition of $[0,1]$ and in fact, we have the following proposition
%which can be read from \cite{jpfb09}.
%%\begin{proposition}\cite{jpfb09}
%For $i\geq1$, let $W_i$ independent uniform $[0,1]$ variables and define  
%\eq
%\lb{stickb}
%J_1:=  W_1, ~~J_2:= (1-W_1) W_2, ~~J_3:= (1-W_1) (1 - W_2) W_3, \ldots 
%\en
%Then the two partitions $\{J_i\}_{i\in\IN}$ and $\{L_i^*\}_{i\in\IZ}$ have the same distribution.
%\end{proposition}
%\begin{proof}

The next proposition clearly states a result we implicitly obtained in the proof of Theorem \ref{bmpp}.  It can be obtained by performing the integration in \eqref{Lapta}, 
but we also provide an independent proof.
\begin{proposition}
Let $(\unC(u), 0 \leq t \leq \Gamma_1)$ be the convex minorant of a Brownian motion on $[0, \Gamma_1]$ and let $\unC'(u)$ denote 
the right derivative of $\unC$ at $u$.  
For $\tau_a=\inf\{u>0:\unC'(u)>a\}$
as in the proof of Theorem \ref{bmpp} and $t>-1$, we have
\begin{align}
\IE e^{-t\tau_a}&=\frac{\sqrt{2+a^2}-a}{\sqrt{2+a^2+2t}-a}.\notag %\label{lapta2}
\end{align}
\end{proposition}
\begin{proof}
Let
$\IE_a$ denote expectation with respect to a BM with drift $-a$ killed at $\Gamma_1$, and $M$ and $T_M$ denote respectively the minimum and time of the minimum of a given process (understood from context).  We now have
\begin{align}
\IE &e^{-t\tau_a}=\IE_a e^{-tT_M} \notag \\
&= \IE_0 \exp\left\{-tT_M-aB(\Gamma_1)-a^2\Gamma_1/2 \right\} \notag \\
&= \IE_0 \exp\left\{ \left(-t-\frac{a^2}{2}\right)T_M-aM-a(B(\Gamma_1)-M)-\frac{a^2}{2}(\Gamma_1-T_M) \right\} \notag \\
&= \IE_0 \exp\left\{ \left(-t-\frac{a^2}{2}\right)T_M-aM\right\} \notag \\
	&\hspace{1in}\times\IE_0\exp\left\{-a(B(\Gamma_1)-M)-\frac{a^2}{2}(\Gamma_1-T_M) \right\},   \notag %\label{lapgi}
\end{align}
where the second equality is a consequence of Girsanov's Theorem (as stated in Theorem 159 of \cite{Freedman1983} under Wald's identity), and the last by Denisov's decomposition at the minimum (specifically independence between the pre and post minimum processes).

Proposition \ref{denvar}
implies that both of $T_M$ and $\Gamma_1-T_M$ are distributed as $\Gamma_{1/2}$, and both of $-M$ and $B(\Gamma_1)-M$ are distributed 
as $\sqrt{\Gamma_{1/2}} R$, with $R$ an independent Rayleigh random variable.  The proposition now follows from Lemma \ref{charden} below.
\end{proof}
\begin{lemma}\label{charden}
If $R$ is Rayleigh distributed and $\Gamma_{1/2}$ has a Gamma$(1/2)$ distribution and the two variables are independent, then for $\alpha<1$ and $(2\alpha+\beta^2)<2$,
\ba
\IE \exp\left\{\alpha\Gamma_{1/2}+\beta\sqrt{\Gamma_{1/2}}R\right\}=\frac{1}{\sqrt{1-\alpha}-\frac{\beta}{\sqrt{2}}}.
\ee
\end{lemma}
\begin{proof}
We have
\ban
\IE \exp&\left\{\alpha\Gamma_{1/2}+\beta\sqrt{\Gamma_{1/2}}R\right\}=\int_0^\infty\frac{e^{-t}e^{t\alpha}}{\sqrt{\pi t}}
\int_0^\infty r e^{-r^2/2}e^{\beta\sqrt{t}r} dr dt \notag \\
&=\int_0^\infty\frac{e^{-t}e^{t \alpha}}{\sqrt{\pi t}}\left[1+\frac{\beta\sqrt{t\pi}}{\sqrt{2}}e^{\beta^2 t/2}\left(1+\erf(\beta\sqrt{t/2})\right)\right] dt, \label{thrints}
\ee
where 
\ba
\erf(x)=\frac{2}{\sqrt{\pi}}\int_0^xe^{-z^2}dz.
\ee
The expression \eqref{thrints} can be broken into the sum of three integrals of which the first
two can be handled by the elementary evaluation
\ban
\int_0^\infty \frac{e^{-tc}}{\sqrt{\pi t}}=c^{-1/2} \label{gamg}
\ee
for $c>0$.  The final integral can be computed using the fact that for $c+d^2>0$,
\ba
\int_0^\infty e^{-tc}\erf(d\sqrt{t})dt=\frac{d}{c\sqrt{c+d^2}}, 
\ee
which can be shown by applying \eqref{gamg} after an integration by parts, noting that
\ba
\frac{d}{d x}\erf(x)=\frac{2}{\sqrt{\pi}}e^{-x^2}.
\ee
\end{proof}

\section{Sequential description}\label{seq}

%In this section we will first develop descriptions of the \cmin\ of various path fragments of Brownian motion, which will yield new insights into the \cmin\ of a standard Brownian motion in the spirit of
%\cite{brownistan, gboom83, p83}, to which we will then compare these results.

%\paragraph{BES(3) bridges}
 
In this section we will prove a result which contains Theorem \ref{mean}, with notation illustrated by Figure \ref{22}, and then derive some corollaries.
We postpone to Section \ref{cqn} discussion of the relation of these results to the \cmin\ of Brownian motion
(specifically the point process description of Section \ref{pppd}).

\begin{theorem}\label{besbr}
Let $(X(v), 0 \le v \le t)$ be one of the following:
\begin{itemize}
\item A $BES(3)$ bridge from $(0,0)$ to $(t,r)$ for $r>0$.
\item A $BES^0(3)$ process.
\item A Brownian meander of length $t$.
\end{itemize}
Let
$(\unC(v), 0 \le v \le t)$ be the \cmin\ of $X$ and let the vertices of $\unC(v)$ occur at times
$0 = V_0 < V_1 < V_2  < \cdots$ with $\lim_n V_n = t$. Let $\tau_n:= t - V_n$ so $\tau_0 = t > \tau_1 > \tau_2 > \cdots$
with $\lim_n  \tau_n  = 0$. Let $\rho_0 = X(t)$ and for $n\geq1$ let $\rho_0- \rho_n$ denote the intercept at time $t$ of the line extending the segment of the
\cmin\ of $X$ on the interval $(V_{n-1}, V_n)$. The \cmin\ $\unC$ of $X$ is uniquely determined by the sequence
of pairs $(\tau_n, \rho_n)$ for $n = 1, 2, \ldots$ which satisfies the \tr\ recursion with
\eq
\rho_0 = X(t) \mbox{ and } \tau_0 = t.
\en
Moreover, conditionally given $(\unC(v), 0 \le v \le t)$ the process $(X(v) - \unC(v),  0 \le v \le t)$ is a concatenation of
independent Brownian excursions of lengths $\tau_{n-1} - \tau_{n}$ for $n \ge 1$.
\end{theorem}

%As variations of Theorem \ref{besbr}, there are the following corollaries for an unconditioned $BES^0(3)$ process and for a Brownian meander.
%\begin{prp}{bes3me}
%Let $(X(v), 0 \le v \le t)$ be either $BES^0(3)$ process, or a Brownian meander of length $t$.
%Then the conclusions of the previous proposition hold, with the fixed value $\rho_0 = r$ replaced
%by the random value $\rho_0 = X(t)$, which is independent of the two sequences $(U_n)$ and $(Z_n^2)$ involved in the representation.  For $\Gamma_s$ a gamma
%random variable as defined in the introduction, we have that
%$\rho_0^2 \ed 2t\Gamma_{3/2}$ in the case of $BES^0(3)$ and $\rho_0^2 \ed 2 t \Gamma_1 $  in the case of a meander of length $t$. 
%\end{prp}
%
%\proof
%This is immediate for the unconditioned $BES^0(3)$ process. For the Brownian meander of length $t$, we appeal to the result of
%Imhof \cite{imhof84} already mentioned in Section \ref{background} that the law of the Brownian meander of length $t$ is absolutely continuous with respect to that of the
%unconditioned $BES^0(3)$ process on $[0,t]$.
%\endpf
%

%\paragraph{The concave majorant of a Brownian first passage bridge}
Before proving the theorem, we note that by essentially rotating and relabeling Figure \ref{22}, we obtain the following
description of the \cmaj\ of a Brownian first passage bridge which is proved by applying Proposition \ref{fpbes} and Theorem~\ref{besbr}. 
\begin{corollary}\label{firstpassage}
Fix $\rho_0 = r > 0$ and let $\rho_1 > \rho_2 >  \ldots > 0$ be the intercepts at $0$ of the linear extensions of segments of the
\cmaj\ of $(B(t), 0 \le t \le \sigma_r)$ where $\sigma_r:= \inf \{t : B(t) = r \}$, and let
$\tau_0  = \sigma_r > \tau_1 > \tau_2 > \cdots$ denote the decreasing sequence of times $t$ such that $(t, B(t))$ is a vertex of the \cmaj\ of $(B(t), 0 \le t \le \sigma_r)$.
Then the sequence of pairs follows the \tr\ recursion with $\rho_0$ as above and $\tau_0 = \sigma_r$.
Moreover, if $(\ovC_r(t), 0 \le t \le \sigma_r)$ denotes the \cmaj, then conditionally given the \cmaj\ the difference process
$(\ovC_r(t) - B(t), 0 \le t \le t \le \sigma_r)$ is a succession
of independent Brownian excursions between the zeros enforced at the times $\tau_n$  of 
vertices of $\ovC$.
\end{corollary}

\begin{proof}[Proof of Theorem \ref{besbr}]
We first prove the theorem for $X$ a $BES(3)$ bridge from $(0,0)$ to $(t,r)$.
Let $(R_3(u), u \ge 0)$ be a $BES^0(3)$ process.
The linear segment
of the \cmin\ of $(R_3(u), 0 \leq u \leq 1)$ connected to zero has slope
$\min_{0 < u \le 1} R_3(u)/u$.  
From the
description of $R_3$ in terms of three independent Brownian motions,
$R_3$ shares the invariance property under time inversion. That is,
\[
%\lb{inv4}
R_3(u) = u \hR(1/u)    \mbox{ where }  0 < u \le  1 \le 1/u
\]
for another $BES^0(3)$ process $\hR$. Observe that for each $a \ge 0$
and $0 < u \le 1$ there is the identity of events
$$
(R_3(u ) \ge a u ) = (\hR(1/u) \ge a )
$$
and hence
$$
(R_3(u ) \ge a u \mbox{ for all } 0 \le u \le 1 ) = (\hR(t) \ge a \mbox{ for all } t \ge 1).
$$
The first item of Proposition \ref{willdecomp} states that the minimum value of a $BES^r(3)$ process has uniform distribution on $[0,r]$, so that given $\hR(1) = r$ the facts above can be applied
to the $BES^r(3)$ process $\hR(1 +s), s \ge 0$ to conclude that
\eq
\min_{0 < u \le 1} R_3(u)/u =  U R_3(1) \label{fsl}
\en
where $U$ is independent of $R_3$, and $U$ has uniform 
distribution of $[0,1]$.  Thus, we conclude that the slope of the first segment of the \cmin\ of a $BES^0(3)$ process on $[0,1]$ has distribution given by \re{fsl}.

Now, if $V_1$ denotes the almost surely unique time $u$ at which
$R_3(u)/u$ attains its minimum on $(0,1]$, then the first
vertex after time $0$ of the convex minorant of $(R_3(u), 0 \le u \le 1)$
is $(V_1, V_1 U R_3(1))$ for $U$ and $R_3(1)$ as above.  We can derive the distribution of $V_1$ by using the Williams decomposition of Proposition
\ref{willdecomp} and Brownian scaling.  More precisely, the second item of Proposition \ref{willdecomp} implies that the distribution of $V_1$ conditioned on $R_3(1)$ and $U$ is
$1/(1 + R_3(1) ^2 (1-U)^2 T_1 )$, where $T_1$ is the hitting time of $1$ by a standard Brownian motion $B$, assumed independent of $R_3(1)$ and $U$.  From this point, we have that
$$
V_1  \ed \frac{ 1 } { 1 + R_3(1) ^2 (1-U)^2 / B(1)^2 } ,
$$
where we have used the basic fact that $T_1\ed B(1)^{-2}$.

The previous discussion implies the the assertions of the theorem about the first face of the \cmin, so we now focus
on determining the law of the process above this face. 
Given $U R_3(1)=a$ and $V_1=v$, the path $(X_1(u), 0 \leq u\leq v) = (R_3(u) - u a, 0 \le u \le v)$ satisfies
$X_1(u)=u(\hR(1/u)-a)$ for $0<u\leq v $, and the latter process is the time inversion of the $BES^0(3)$ process appearing in the third item of the
Williams decomposition of Proposition
\ref{willdecomp}.  Under this conditioning,
$(X_1(u), 0 \leq u\leq v)$ is a $BES^0(3)$ process conditioned to be zero at time $v$, which implies $X_1$
is a Brownian excursion of length $v$.
Similarly, given $R_3(1) =r$, $V_1 = v$, $R_3(V_1) = a v$
the process $(R_3(v + w) - ( R_3(v) + a w ) , 0 \le w \le 1 - v)$  is a $BES(3)$ bridge from $(0,0)$ to $(1-v, r- a)$, and
after a simple rescaling, this decomposition can be applied again to the
remaining $BES(3)$ bridge from $(0,0)$ to $(1-v, r- a)$, to recover the second segment of the \cmin\ of  
$(R_3(u) , 0 \le u \le 1)$, and so on.  
With Brownian scaling, this proves the result for a $BES(3)$ bridge.

Finally, the result follows immediately for the unconditioned $BES^0(3)$ process,
and for the Brownian meander of length $t$, we appeal to the result of
Imhof \cite{imhof84} given previously as Proposition \ref{imh} that the law of the Brownian meander of length $t$ is absolutely continuous with respect to that of the
unconditioned $BES^0(3)$ process on $[0,t]$ with density depending only on the final value.
\end{proof}

\section{Consequences}\label{cqn}

We now return to the discussion related to Theorem \ref{equiv} surrounding the relationship between our two descriptions.
First notice that the Poisson point process description for Brownian motion on the interval $[0, \Gamma_1]$ yields an analogous
description for a meander of $\Gamma_{1/2}$ length by restricting the process to positive slopes.  This observation yields the following Corollary of Theorem~\ref{bmpp}.  Note that we have introduced a factor of two in the length of the meander to simplify the formulas found below.
\begin{corollary}\label{mecor}
Let $(M(t),0\leq t \leq 2\Gamma_{1/2})$ be a Brownian meander of length $2\Gamma_{1/2}$.
Then the lengths $x$ and slopes $s$ of the faces of the \cmin\ of $M$ form a Poisson point process on $\IR^+ \times \IR^+$
with intensity measure
\begin{align}
\frac{\exp\{-\frac{x}{2}\left(1+s^{2}\right)\}}  {\sqrt{2 \pi x}}  \, ds\,dx, \hspace{5mm} x,s \geq0. \label{imbme}
\end{align}
\end{corollary}
\begin{proof}
Denisov's decomposition implies that $M$ can be constructed as the fragment of a Brownian motion $B$ on $[0, 2\Gamma_1]$, occurring after the time of the minimum.  Since the minimum of a Brownian motion on $[0,1]$
occurs at an arcsine distributed time and the faces of the \cmin\ of $B$ after
the minimum are simply the faces with positive slope, the corollary follows from Theorem \ref{bmpp} and Brownian scaling. 
\end{proof}
\begin{remark}
By scaling out the meander by a factor of two, the density \eqref{imbme} differs only slightly from \eqref{imbm}.  In general, the Poisson point
process of lengths
$x$ and slopes $s$ of the \cmin\ of
a Brownian motion on $[0, \theta\Gamma_1]$ has density
\ba
\frac{\exp\{-\frac{x}{2}\left(\frac{2}{\theta}+s^{2}\right)\}}  {\sqrt{2 \pi x}} \, ds\,dx, \hspace{5mm} x\geq0, s\in \IR,
\ee
which follows from Brownian scaling. 
\end{remark}

Alternatively, the construction of Theorem \ref{mean} implies that we can in principle obtain the lengths and slopes of the \cmin\ of $M$ through
the variables $\{(\tau_i, \rho_i), i=0, 1, \ldots\}$ as illustrated by Figure \ref{22}.  Precisely, we have the following result
which follows directly from Theorem \ref{besbr} and the definition of a meander of a random length given in Section \ref{background}.
\begin{corollary}\label{xxi}
Using the notation from Figure \ref{22}, let $(2\Gamma_{1/2}-\tau_i,\rho_i)$ be the times of the vertices and the intercepts of the \cmin\ of 
$(M(t),0 \leq t \leq 2\Gamma_{1/2})$, a Brownian meander of length $2\Gamma_{1/2}$.
Then the sequence $(\tau_i,\rho_i)$ follows the \tr\ recursion with
$\tau_0\ed2\Gamma_{1/2}$ and $\rho_0\ed\sqrt{\tau_0}R$, where $R$ has the Rayleigh distribution.  
\end{corollary}

The descriptions of Corollaries \ref{mecor} and \ref{xxi} are defining in the sense that either one in principle is derivable from the other.  However, it is not obvious how to implement this program, and moreover, even some simple equivalences elude
independent proofs.  In the remainder of this section we will explore these equivalences.  

\begin{proposition}\label{uui}
\mbox{}
\begin{enumerate}
\item
Let $(V_1,S_1)$ denote the length and slope of the segment with the minimum slope of the \cmin\ of $M$ as defined in Corollary \ref{mecor}. Then
\ban
\IP(V_1\in dv, S_1 \in da)=v^{-{1/2}}e^{-v/2}\phi(a\sqrt{v})(\sqrt{1+a^2}-a) dv da, \label{laden}
\ee
where $\phi(x)=(2\pi)^{-1/2}e^{-x^2/2}$ is the standard normal density.
\item If $W$ is a standard normal random variable independent of $S_1$, then
\ban\label{fact}
\left(V_1(1+S_1^2),S_1\right) \ed \left(W^2, S_1\right).
\ee
\end{enumerate}
\end{proposition}
\begin{proof}
From the Poisson description of Corollary \ref{mecor}, 
\ba
\IP(V_1\in dv, S_1 \in da) = v^{-{1/2}}e^{-v/2}\phi(a\sqrt{v}) \times P_0(a),
\ee
where $P_0(a)$ is the chance of having no points of the Poisson process with slope less than $a$.
Now,
\begin{align}
P_0(a)=\IP(S_1>a) &= \exp\left\{-\int_0^a \int_0^\infty v^{-{1/2}}e^{-v/2}\phi(s\sqrt{v}) dv ds\right\}\notag\\
&= \exp\left\{-\int_0^a(1+s^2)^{-1/2}da \right\} \notag \\
&=\sqrt{1+a^2}-a, \label{sdist}
\end{align}
which implies the first item of the proposition.

The second item follows after making the substitution $t=v(1+a^2)$ in \eqref{laden}.
%\begin{align}
%\IP( V_1(1+S_1^2) \in dt, S_1 \in da) &= \frac{e^{-t/2}}{\sqrt{2\pi t}}\left(1-\frac{a}{\sqrt{1+a^2}}\right) dt da \notag\\
%&= \IP(2\Gamma_{1/2} \in dt) \IP(S_1 \in da). \notag
%\end{align}
\end{proof}

Comparing Proposition \ref{uui} with the analogous conclusions of Corollary \ref{xxi} yields the following remarkable identity.
\begin{theorem}\label{remid}
Let $R$ Rayleigh distributed, $U$ uniform on $[0,1]$,  $Z$ and $W$ standard normal, and $T\ed2\Gamma_{1/2}$ be independent random variables.  If $\barU:=1-U$, then
\eq
\lb{ident3}
\left( \frac{ T + \barU ^2 R^2 }{ 1 + U^2 R^2/Z^2 }, \frac{ \barU  R } {\sqrt{T}} \right)
\ed \left( W^2 , S_1 \right),
\en
where on the right side the two components are independent (hence also on the left). 
\end{theorem}
\begin{proof}
Because the face of the \cmin\ with minimum slope is also the first face, we know that 
\ban
(V_1,S_1)\ed(\tau_0-\tau_1, (\rho_0-\rho_1)/\tau_0), \label{rhid}
\ee
where the sequence $(\tau_i, \rho_i)_{i\geq0}$ is defined as in Corollary \ref{xxi}.  Corollary \ref{xxi} also implies that we have the representation
\ba
(\tau_0, \rho_0) = (T, \sqrt{T} R)
\ee
and
\ba
(\tau_1, \rho_1)=\left(\frac{U^2R^2T}{Z^2+U^2R^2}, U\sqrt{T}R\right),
\ee
so that using \eqref{rhid} we find
\ban\label{meed}
(V_1,S_1)\ed\left(\frac{TZ^2}{Z^2+U^2R^2}, \frac{R\barU}{\sqrt{T}}\right).
\ee
Combining \eqref{fact} and \eqref{meed} yields the theorem.
\end{proof}

\begin{remark}
The straightforward calculation 
\begin{align}
\IP\left(\frac{\barU R}{\sqrt{T}}>s\right)&=\sqrt{\frac{2}{\pi}}\int_0^\infty\int_{st}^\infty\int_{st/r}^1re^{-r^2/2}e^{-t^2/2}dudrdt \notag \\
			&=\sqrt{1+s^2}-s, \label{rtudist}
\end{align}
shows that the distribution of the second component on the left hand side of \eqref{ident3} agrees with that on that right given by \eqref{sdist},
but the equality in distribution
of first components given by Corollary \ref{bewid} of the introduction
is not as obvious.
\end{remark} 

\begin{proposition}\label{con1}
\mbox{}
\begin{enumerate}
\item If $(L_i,S_i)$ is the length and slope of the $i$th face of the \cmin\ of $M$ (with $T\ed 2\Gamma_{1/2}$ as above), then
\begin{align}
\IP&(L_i\in dx, S_i \in da) \notag \\
&=x^{-{1/2}}e^{-x/2}\phi(a\sqrt{x})(\sqrt{1+a^2}-a)\frac{\left(-\log(\sqrt{1+a^2}-a)\right)^{i-1}}{(i-1)!} dx da. \label{laden1}
\end{align}
\item If $W$ is a standard normal random variable independent $S_i$, then
\begin{align}
(L_i(1+S_i^2),S_i)\ed (W^2, S_i), \label{itind}
\end{align}
\end{enumerate}
\end{proposition}
\begin{proof}
Similar to the proof of Proposition \ref{uui},
\ba
\IP(L_i\in dx, S_i \in da) = x^{-{1/2}}e^{-x/2}\phi(a\sqrt{x}) \times P_{i-1}(a),
\ee
where $P_{i-1}(a)$ is the chance of having $i-1$ points of the Poisson process with slope less than $a$.
Since the number of points with slope less than $a$ is a Poisson random variable with mean
$-\log(P_0(a))$, the first item follows.

The second item is immediate after making the substitution $t=x(1+a^2)$ in \eqref{laden1}. 
\end{proof}
\begin{remark}
Integrating out the variable $x$ in \eqref{laden1} implies 
\ba
\IP(S_i\in da)=\left(1-\frac{a}{\sqrt{1+a^2}}\right)\frac{\left(-\log(\sqrt{1+a^2}-a)\right)^{i-1}}{(i-1)!} da,
\ee
while the marginal density for $L_i$ does not appear to simplify beyond the expression obtained by
integrating out $a$ in \eqref{laden1}.
%\ban
%\frac{\IP(L_i\in dx)}{dx}=\int_0^\infty x^{-{1/2}}e^{-x/2}\phi(a\sqrt{x})(\sqrt{1+a^2}-a)\frac{\left(-\log(\sqrt{1+a^2}-a)\right)^{i-1}}{(i-1)!}  da. \label{latt}
%\ee
\end{remark}
Alternatively, we can use the sequential description to obtain the following in the case where $i=2$ (noting that $L_i=\tau_{i-1}-\tau_i$).
\begin{proposition}\label{con2}
For $i=1,2$ let $Z_i$ be independent standard normal random variables and $U_i$ independent uniform $(0,1)$ random variables.
Then
\ban
S_2\ed\frac{(1-U_1U_2)R\sqrt{T}-V_1S_1}{T-V_1}=\frac{R}{\sqrt{T}}\left(1-U_1U_2+\frac{Z_1^2(1-U_2)}{U_1R^2}\right), \label{eq1}
\ee
and
\ban
L_2\ed\frac{TU_1^2R^2Z_2^2}{(Z_1^2+U_1^2R^2)(Z_2^2+(Z_1^2+U_1^2R^2)U_2^2)}. \label{eq2}
\ee
Moreover, the equivalences given by \eqref{eq1} and \eqref{eq2} hold jointly.  
\end{proposition}

Combining Propositions \ref{con1} and \ref{con2} would yield a result similar to, but more complicated than Theorem \ref{remid}.  Moreover,
it is not difficult to obtain more identities by considering greater indices.
These identities seem to defy independent proofs.  We leave it as an open problem to construct a framework to explain
these equivalence without reference to Brownian motion.

\section{Density Derivations}\label{ppc}

In this section we use Corollary \ref{mecor} to derive various densities and transforms associated to the process of vertices and slopes of faces of the \cmin\ of Brownian motion and meander.  First we define the inverse hyperbolic functions
\ba
\arcsinh(x)&:=\log \left( x+\sqrt{1+x^2} \right), \hspace{5mm} x\in\IR,\\
\arcosh(x)&:=\log \left( x+\sqrt{x^2-1} \right), \hspace{5mm} x \geq1
\ee
and to ease notation, let 
\ba
\ar(t):=\arcosh(t^{-1/2}), \hspace{5mm} 0<t\leq1.
\ee
\begin{theorem}\label{denstease}
Using the notation of Theorem \ref{mean} and Figure \ref{22} with $t=1$, for $n=1, 2, \ldots$ let $1-\tau_n$ be the time of the right endpoint of the $n$th face of the \cmin\ of a standard Brownian meander, and let $f_{\tau_n}$ denote the density of
$\tau_n$.  For $0< t < 1$, and $|z|<1$, we have
\ban
\sum_{n=1}^\infty f_{\tau_n}(t)z^n=\left(\frac{1}{2(1-t)^{3/2}}\right)\frac{z\left[-1+\left(\frac{1-z\sqrt{1-t}}{\sqrt{t}}\right)\left(\frac{\sqrt{1-t}+1}{\sqrt{t}}\right)^z\right]}{(1-z^2)}.\label{genft}
\ee
In the case $z=1$, we obtain
\begin{align}
\sum_{n=1}^\infty f_{\tau_n}(t)=\frac{1-t +\sqrt{1-t}-t\,\ar(t)}{4t(1-t)^{3/2}}, \label{intsmes}
\end{align}
which is the intensity function of the (not Poisson) point process with points $\{\tau_n: n\in\IN\}$.

\end{theorem}

Before proving the theorem, we record some corollaries.
\begin{corollary}
For $n\geq1$,
\ban
f_{\tau_n}(t)=\frac{1}{4(1-t)^{3/2}}\sum_{k=1}^{\infty}\left(1-(-1)^{n+k}\right)\binom{k-1}{n-1}\frac{a(t)^k}{k!}. \label{tden}
\ee
\end{corollary}
\begin{proof}
Let 
\ba
e_n(t):=\sum_{k=n}^\infty\frac{t^k}{k!}=e^t-\sum_{k=0}^{n-1}\frac{t^k}{k!}, 
\ee
and $h_n(t):=e^t(-1)^n e_n(-t)$. 
By considering series expansion on the right hand side of \eqref{genft}, a little bookkeeping leads to
\ba
f_{\tau_n}(t)=\frac{h_n(a(t))-(-1)^nh_n(-a(t))}{4(1-t)^{3/2}}.
\ee
The corollary now follows after noting
\ba
h_n(t)=\sum_{k=1}^\infty \binom{k-1}{n-1}\frac{t^k}{k!},
\ee
which can be proved by equating coefficients in the identity
\ba
\sum_{n=1}^\infty  h_n(t)x^n=\frac{x}{1+x}\left(e^{t(1+x)}-1\right),
\ee
or read from \cite{as} (Section 6.5, equations 4, 13, and 29).
\end{proof}

Due to the relationship between Brownian motion and meander elucidated in the introduction, we can obtain results analogous to those above for
Brownian motion on a finite interval.  
\begin{corollary}\label{alpim}
Let $(\alpha_i)_{i\in\IZ}$ be the times of the vertices of the \cmin\ of a Brownian motion on $[0,1]$ as described in the introduction by \eqref{alps} and \eqref{alp0}.
If
$f_{\alpha_i}$ denotes the density of $\alpha_i$ for $i\in\IZ$,
then
\ban
\sum_{i\in\IZ} f_{\alpha_{i}}(t)=\frac{1}{2t(1-t)}, \label{bmzo}
\ee
which is the intensity function of the (not Poisson) point process of times of vertices of the \cmin\ of Brownian motion on $[0,1]$.
\end{corollary}
\begin{proof}
Since $\alpha_n\ed1-\alpha_{-n}$, observe that 
\ban
\sum_{i\in\IZ} f_{\alpha_{i}}(t)=\sum_{i=1}^{\infty}f_{\alpha_{-i}}(t)+\sum_{i=1}^{\infty}f_{\alpha_{-i}}(1-t)+f_{\alpha_0}(t), \label{wwwww}
\ee
and that $\alpha_0$ has the arcsine distribution so that
$f_{\alpha_0}(t)=1(\pi\sqrt{t(1-t)})$.
We will show
\ba
\sum_{i=1}^{\infty}f_{\alpha_{-i}}(t)=\frac{1}{4t}+\frac{1}{2\pi}\left(\frac{\arcos(\sqrt{t})}{t(1-t)}-\frac{1}{\sqrt{t(1-t)}}\right),
\ee
which after substituting and simplifying in \eqref{wwwww}, will prove the corollary.

Since $\alpha_{-i}\ed\alpha_0\tau_i$, with $\alpha_0$ and $\tau_i$ independent, we have
\ban
\sum_{i=1}^{\infty}f_{\alpha_{-i}}(t)&=\sum_{i=1}^{\infty}f_{\alpha_0\tau_i}(t) \notag \\
&=\frac{1}{\pi} \sum_{i=1}^{\infty}\int_t^1v^{-3/2}(1-v)^{-1/2}f_{\tau_i}(t/v)dv \notag \\
&=\frac{1}{\pi} \int_t^1v^{-3/2}(1-v)^{-1/2} \left(\sum_{i=1}^{\infty}f_{\tau_i}(t/v)\right) dv, \label{hwt}
\ee
where the second equality is due to the arcsine density of $\alpha_0$, and the last by Fubini's theorem.

The sum in \eqref{hwt} can be evaluated using \eqref{intsmes} of Theorem \ref{denstease}, and the corollary will follow
after evaluating the integral in \eqref{hwt}.  There is some subtlety in carrying out this integration, so we refer to 
the appendix for the relevant calculations.  
\end{proof}
\begin{remark}
The method of proof of Corollary \ref{alpim} can be used to obtain an expression for $f_{\alpha_{-i}}(t)$, 
for $i\in\IN$.  For example,
\eqref{tden} implies that
\ba
f_{\tau_1}(t)&=\frac{1}{2(1-t)^{3/2}}\left(t^{-1/2}-1\right) \\
	&= \frac{1} {2  } \sum_{n=1}^\infty \frac{ (\hf)_{n} ( 1 - t)^{n-\thf} }{n!}, 
\ee
where $(a)_n=a(a+1)\cdots(a+n-1)$.
Using the Proposition \ref{asiv} of the appendix, we find
\ba
f_{\alpha_{-1}}(t)&=\frac{1}{2\pi} \int_t^1v^{-3/2}(1-v)^{-1/2}f_{\tau_1}(t/v)dv \\
	&=\frac{1} {2 \sqrt{\pi t} } \sum_{n=1}^\infty 
\frac{ (\hf)_{n} \Gamma(n - \hf)( 1 - t)^{n-1} }{n! (n-1)! }.
\ee
As the index~$i$ increases, these expressions become more complicated, but it is in principle possible to obtain expressions for $f_{\alpha_{-i}}$ by
expanding $f_{\tau_i}$ appropriately.

%Also, it can be shown using some facts from Section \ref{cqn} that for $i\geq1$,
%\ba
%\IE\tau_i=2(1-\beta(i)),
%\ee
%where
%\ba
%\beta(i)=\sum_{n = 0}^\infty (-1)^n ( 2 n + 1 )^{-i}
%\ee
%is the \emph{Dirichlet beta function} (page 807 in \cite{as}).
%Also note that by Proposition \ref{lighted},
%\ba
%\IE ( \alpha_{n} ) = 1- \IE ( \alpha_{0}  ) \E( \tau_n  ), 
%\ee
%which implies that 
%\ba
%\IE ( \alpha_{n} ) = \beta(n).
%\ee
%We refer to the appendix for further discussion on the relationship between the sequence $(\alpha_i)_{i\in\IZ}$ and $(\tau_n)_{n\geq0}$.
\end{remark}

%\begin{remark}
%By considering series expansions, we have that for $k\geq 0$
%\begin{align}
%f_{\tau_{2k+1}}(t)=\frac{1}{2(1-t)^{3/2}}\Bigg[-1+\frac{1}{\sqrt{t}}\sum_{j=0}^{k}\frac{\ar(t)^{2j}}{(2j)!} 
%-\sqrt{\frac{1-t}{t}}\sum_{j=1}^{k}\frac{\ar(t)^{2j-1}}{(2j-1)!}\Bigg] \notag
%\end{align}
%and for $k\geq 1$
%\begin{align}
%f_{\tau_{2k}}(t)=\frac{1}{2(1-t)^{3/2}}\Bigg[\frac{1}{\sqrt{t}}\sum_{j=1}^{k}\frac{\ar(t)^{2j-1}}{(2j-1)!} 
%&-\sqrt{\frac{1-t}{t}}\sum_{j=0}^{k-1}\frac{\ar(t)^{2j}}{(2j)!}\Bigg], \notag
%\end{align}
%where we use the convention $\sum_a^b=0$ if $b<a$.
%\end{remark}
\begin{corollary}
The point process of times of vertices of the \cmin\ of Brownian motion on $[0, \infty)$ has intensity function $(2u)^{-1}$.
\end{corollary}
\begin{proof}
From \cite{jpfb09}, the process of times of vertices of Brownian motion on $[0,1]$ has the distribution of the analogous
process for standard Brownian
bridge.  Also, the Doob transform which maps standard Brownian bridge to infinite horizon Brownian motion
preserves vertices of the \cmin.  Thus, we apply the time change of variable $u=t/(1-t)$ of the Doob transform
to \eqref{bmzo} of Corollary \ref{alpim} which yields the result.
\end{proof}

Now, in order to prove Theorem \ref{denstease}, we consider the \cmin\ of a meander of length a $2\Gamma_{1/2}$ random variable as
the faces of positive slope of the \cmaj\ of a Brownian
motion on $[0, 2\Gamma_1]$ similar to Corollary~\ref{mecor}. 
We collect the following facts.
\begin{lemma}
Let $B$ a Brownian motion, $(\ovC_t, 0\leq t \leq 2\Gamma_1)$ be the \cmaj\ of $B$ on $[0, 2\Gamma_1]$, and $\ovC_t'$ denote the right
derivative of $\ovC_t$.  If
\ban
\sigma_u:=\sup\{t>0, \ovC_t'\geq 1/u\}, \label{sigu}
\ee
then
\ban
\IE e^{-a \sigma_u}=\frac{1+\sqrt{1+u^2}}{1+\sqrt{1+u^2+2a u^2}}. \label{siglap}
\ee
\end{lemma}
\begin{proof}
We make the change of variable $a=1/u$ in the Poisson process intensity measure given by \eqref{imbm}, so 
that the intensity measure of the lengths and inverses of positive slopes of $\ovC_t$ is given by
\ban
\frac{\exp\{-\frac{t}{2}\left(1+u^{-2}\right)\}}  {u^2 \sqrt{2 \pi t}}   \, dt\,du, \hspace{5mm} t,u \geq0. \label{tt2}
\ee

The lemma follows after noting that $\sigma_u$ can alternatively be defined as the sum
of the lengths of the points of the Poisson point process given by \eqref{tt2} with inverse slope smaller than $u$, so that
\ba
\IE e^{-a \sigma_u}&=\exp\left\{-\int_0^\infty\left(1-e^{-a t}\right)\int_0^u\frac{1}{v^2\sqrt{2\pi t}}\exp\left\{-\frac{t}{2}\left(1+v^{-2}\right)\right\}dv dt\right\}.
\ee
\end{proof}

Because the segments of the \cmaj\ of $B$ appear in order of decreasing slope, it will be useful for the purpose of tracking indices to first discuss the number
of segments with slope smaller than a given value.  
\begin{lemma}\label{inslop}
The intensity function of the Poisson point process of inverse slopes $u$ of $\ovC$, the \cmaj\ of
a Brownian motion on $[0, 2\Gamma_1]$, is
\ba
\lambda(u):= \frac{ 1  } { u \sqrt{ 1 +  u^2 } }.
\ee
The number of segments of $\ovC$ with slope smaller than $1/u$ is a Poisson random variable
with mean
\ban
\Lambda(u) :=  \int_u^{\infty} \lambda(v) dv  = \arcsinh (u^{-1}). \label{lamb}
\ee
\end{lemma}
\begin{proof}
Integrating out the lengths $t$ from \eqref{tt2}
%using
%$$
%\int_0^\infty \frac{ \exp ( - a t ) } { \sqrt { \pi t } } = a^{-1/2},
%$$
yields the intensity $\lambda(u)$ and the second statement is evident from the first.
\end{proof}

Define $T_0$ to be the time of the maximum of $B$ on $[0,2\Gamma_1]$ and for $n=1, 2, \ldots$, let $T_n$ be the time of the left endpoint of the the face of the \cmaj\ with $n$th smallest positive slope.  Note that $T_0>T_1>\ldots$  and that Brownian scaling implies that $T_n\ed2\Gamma_{1/2}\tau_n$.  Our basic strategy is to obtain information about
the~$T_n$ and then ``de-Poissonize" in order to yield
analogous information for the~$\tau_n$.  

\begin{proposition}\label{Tn}
Let $f_{T_n}$ denote the density of $T_n$.  Then
\ba
f_{T_n } (t) =\frac{e^{- t/2}}{2}
\int_0^\infty \frac{\arcsinh^n (v )}{n!} \, \erfc \left( v\sqrt{t/2} \right)\, dv,
\ee
where
\ba
\erfc(x)=\frac{2}{\sqrt{\pi}}\int_x^\infty e^{-r^2}dr=\IP(Z^2>2x^2). 
\ee
\end{proposition}
\begin{proof}
For each $n$ we can find the distribution of $T_n$ by conditioning on the
inverse slope $U_n$ of the segment from $T_{n+1}$ to $T_n$. 
We can obtain such an expression because
$\{(T_{n-1}-T_n,U_{n-1}): n\in\IN\}$ is the collection of points of a Poisson process with intensity measure given by \eqref{tt2},
so that we can write down 
\ba
\frac{\IP(U_n \in du, T_{n+1} \in dv, T_n \in dt)}{du\,dv\,dt} = f_{\sigma_u}(v)  \frac{\exp\{-\frac{(t-v)}{2}\left(1+u^{-2}\right)\}}  {u^2 \sqrt{2 \pi (t-v)}}\frac{ e^{-\Lambda(u) }  \Lambda(u)^n }{n!},
\ee
where we are using Lemma \ref{inslop},  $\Lambda(u)$ is given by \eqref{lamb}, and the definition of $\sigma_u$ is given by \eqref{sigu}.
Integrating out $u$ and $v$ and noting the convolution of densities, the expression above leads to 
\ban
f_{T_n}(t) =  \int_0^\infty \lambda(u) f_{Y_u}(t) \frac{ e^{-\Lambda(u) }  \Lambda(u)^n }{n!} du, \label{tt33}
\ee
where $Y_u\ed Z^2/(1+u^{-2})+\sigma_u$ and $Z$ is a standard normal random variable independent of $\sigma_u$.

We proceed to obtain a more explicit expression for $f_{T_n}$ after determining $f_{Y_u}$ by inverting its Laplace transform. Using \eqref{siglap},
we obtain
\begin{align}
\IE e^{-a Y_u}&=\IE e^{-a \sigma_u}\IE e^{-a Z^2/(1+u^{-2})} \notag \\
	&=\left(\frac{1+\sqrt{1+u^2}}{1+\sqrt{1+u^2+2a u^2}}\right)\left(\frac{\sqrt{1+u^2}}{\sqrt{1+u^2+2a u^2}}\right). \notag 
\end{align}
Inverting this Laplace transform we find that
\ban
f_{Y_u}(t) = \frac{ \sqrt{ 1 + u^2} \,( 1 + \sqrt{ 1 + u^2 } ) }{ 2 u^2 } \erfc \left( \frac{\sqrt{t/2}}{u} \right) e^{-t/2}. \label{yud}
\ee
Combining \eqref{tt33} and \eqref{yud} yields 
\ba
f_{T_n } (t) = \frac{e^{- t/2}}{2}\int_0^\infty \frac{\arcsinh^n (u^{-1})}{n!}  \, \erfc \left( \frac{\sqrt{t/2}}{u} \right) u^{-2} \, du,
\ee
and the result is proved after making the change of variable $u = 1/v$.
\end{proof}

We are now in a position to prove Theorem \ref{denstease}.
\begin{proof}[Proof of Theorem \ref{denstease}]
Proposition \ref{Tn} implies that for $-1<z\leq1$, 
\begin{align}
\sum_{n=1}^\infty& z^n f_{T_n } (t)  \notag \\
&=\frac{e^{- t/2}}{2}
\int_0^\infty\left( \left(v+\sqrt{1+v^2}\right)^z-1\right) \, \erfc \left( v\sqrt{t/2} \right)\, dv. \label{pser}
\end{align}
From this point, the theorem will be proved after de-Poissonizing \eqref{pser} to obtain an analogous expression with $\tau_n$ in place of $T_n$.

Because $T_n\ed 2 \Gamma_{1/2} \tau_n$, Brownian scaling implies
\begin{align}
f_{T_n } (t)&=\int_{t}^\infty f_{x\tau_n}(t)\frac{e^{-x/2}}{\sqrt{2\pi x}} dx \notag \\
&=\int_{t}^\infty f_{\tau_n}(t/x)x^{-1}\frac{e^{-x/2}}{\sqrt{2\pi x}} dx \notag \\
&=\int_{0}^1 f_{\tau_n}(u) \frac{e^{-t/(2u)}}{\sqrt{2\pi t u}} du, \notag 
\end{align}
so that for $-1<z\leq 1$ and $F(z,t)=\sum_{n\geq1}z^n f_{\tau_n}(t)$, we have
\begin{align}
\sum_{n=1}^\infty& z^n f_{T_n } (t) =  \int_{0}^1 F(z, u) \frac{e^{-t/(2u)}}{\sqrt{2\pi t u}} du.\label{g1}
\end{align}
Combining \eqref{pser} and \eqref{g1}, we arrive at the integral equation
\begin{align}
\int_{0}^1 F(z,u) \frac{e^{-t/(2u)}}{\sqrt{2\pi t u}} du = \frac{e^{- t/2}}{2}
\int_0^\infty g(z,v) \, \erfc \left( v\sqrt{t/2}  \right)\, dv \notag
\end{align}
where $g(z, v)=\left(v+\sqrt{1+v^2}\right)^z-1$. 
After simplification, we obtain the following integral equation for $F$
\begin{align}
\int_{0}^1 F(z, u) \frac{e^{-t/(2u)}}{\sqrt{ u}} du = t e^{- t/2} 
\int_0^\infty e^{-tx^2/2}\left[ \int_{0}^xg(z, v)\,dv\right] dx. 
\end{align}
Lemma \ref{inteq11} below indicates the solution to this integral equation and the theorem follows after noting
\[\int_{0}^xg(z, v)\,dv=\frac{\left(x+\sqrt{1+x^2}\right)^z\left(x-z\sqrt{1+x^2}\right)+z}{1-z^2}-x\]
in the case where $|z|<1$, and 
\[\int_{0}^xg(1, v)\,dv=\frac{x\left(x+\sqrt{1+x^2}\right)+\arcsinh(x)}{2}-x.\]
\end{proof}

\begin{lemma}\label{inteq11}
Let $F$ a function on $(0,1)$ and $G$ a differentiable function on $(0,\infty)$
such that
\[\lim_{x\rightarrow0}G(x)/x=0.\]
If 
\begin{align}
\int_{0}^1 F(u) \frac{e^{-t/(2u)}}{\sqrt{ u}} du = t e^{- t/2} 
\int_0^\infty e^{-tx^2/2} G(x) dx, \hspace{5mm} t>0, \label{g2}
\end{align}
with the assumption that the integrals converge for $t>0$, 
then
\ba
F(u)=\frac{1}{2(1-u)^{3/2}}\left[\sqrt{\frac{1-u}{u}}G'\left(\frac{1-u}{u}\right)-G\left(\sqrt{\frac{1-u}{u}}\right)\right].
\ee 
\end{lemma}
\begin{proof}
The change of variable $u=(1+x^2)^{-1}$ on the left hand side of \eqref{g2} yields
\begin{align}
\int_{0}^\infty F((1+x^2)^{-1}) \frac{2xe^{-tx^2/2}}{(1+x^2)^{3/2}} dx = t 
\int_0^\infty e^{-tx^2/2}G(x) dx. \label{ffr}
\end{align}
Notice that the left hand side of \eqref{ffr} is essentially a Laplace transform.  Since
\[\lim_{x\rightarrow0} G(x)/x=0,\]
an integration by parts on the right hand side implies \eqref{ffr} can be written
\begin{align}
\int_{0}^\infty e^{-tx^2/2}  \frac{2xF((1+x^2)^{-1})}{(1+x^2)^{3/2}} dx = 
\int_0^\infty e^{-tx^2/2}\left[ \frac{xG'(x)-G(x)}{x^2}\right] dx. \notag
\end{align}
Uniqueness of Laplace transforms now yields
\[\frac{2xF((1+x^2)^{-1})}{(1+x^2)^{3/2}}=\frac{xG'(x)-G(x)}{x^2},\]
and the lemma follows after making the substitution $u=(1+x^2)^{-1}$.
\end{proof}

\section{Sequential Derivations}\label{seqcon}

As Theorem \ref{equiv} indicates, we can view the \tr\ recursion as a Markov chain independent of the Brownian framework from which
it was derived. 
We have the following fundamental result.
\begin{proposition}\label{stationary}
Let $(\rho_n,\tau_n)$ follow the \tr\ recursion for some arbitrary initial distribution of $(\rho_0,\tau_0)$, and 
let $\rho_n^*:= \rho_n/\sqrt{\tau_n}$ which represents the standardized final value of a Brownian path fragment from $(0,0)$ to $(\tau_n,\rho_n)$.
Whatever the initial distribution $(\rho_0, \tau_0)$, the distribution of $\rho_n^*$ converges in total variation as $n \te \infty$ to the
unique stationary distribution of $\rho_n^*$ for the \tr\ recursion, which is the distribution of $\sqrt{2 \Gamma_{3/2} U}$ where $U$
is a uniform $(0,1)$ random variable independent of $\Gamma_{3/2}$.
\end{proposition}
\begin{proof}
%The \tr\ recursion implies 
%\ban
%\left(\frac{\rho_{n+1}}{\sqrt{\tau_{n+1}}}\right)^2=U_{n}^2\left[\left(\frac{\rho_n}{\sqrt{\tau_n}}\right)^2+\frac{Z_{n+1}^2}{U_{n}^2}\right], \label{stod}
%\ee
%so that the proposition is equivalent to the statement that the Markov chain $X_0, X_1, \ldots$ which follows the rule
%\ba
%X_{n+1}=B_{n+1}(X_n+C_{n+1}),
%\ee
%where $(B_i)_{i\in\IN}$ is an i.i.d. sequence of squared uniform $(0,1)$ variables independent of $(C_i)_{i\in\IN}$, and i.i.d. sequence of
From the definition of the \tr\ recursion, the sequence $(\rho^*_n)_{n\geq0}$ satisfies
\ban
\rho^*_{n+1}=\sqrt{Z_{n+1}^2+U_n^2\left(\rho_n^*\right)^2}, \hspace{5mm} n\geq0, \label{strec}
\ee
where $(U_n)_{n\geq0}$ are i.i.d. uniform $(0,1)$ and $(Z_n)_{n\geq1}$ are i.i.d. standard normal, both independent of $\rho_0^*$.
Thus, the chain $(\rho_n^*)_{n\geq0}$ is Markovian and converges to its unique stationary distribution since it is strongly aperiodic (from any given state, the support of the density of
the transition kernel is the positive half line), and positive Harris recurrent (see Theorem 13.3.1 in \cite{MR1287609}). %%%Meyn and Tweedie, Markov Chains and Stochastic Stability 1993  

The relation \eqref{strec} also implies that in order to show the stationary distribution is as claimed, 
%assume that $\rho_0/\sqrt{\tau_0}$ is distributed as $\sqrt{ 2 \Gamma_{3/2} U }$ and that
%the \tr\ recursion is satisfied.  Using the recursion we obtain 
%\[\frac{\rho_1}{\sqrt{\tau_1}}=\sqrt{Z_1^2+U_1^2\left(\frac{\rho_0}{\sqrt{\tau_0}}\right)^2},\]
we must show that for $S\ed2 \Gamma_{3/2} U$, we have 
\ban
S \ed S U^2 + Z^2, \label{rdist}
\ee
for $U$ uniform $(0,1)$ and $Z$ standard normal, independent of each other and of~$S$.

After some manipulations using beta-gamma algebra, it can be seen that \eqref{rdist} is equivalent to 
\ban
\frac{\Gamma_1}{\Gamma_1+\Gamma_1'} \Gamma_{3/2} \ed \frac{\Gamma_1}{\Gamma_1+\Gamma_1'} \Gamma_{1/2} +\Gamma_{1/2}', \label{bga}
\ee
where all the
variables appearing are independent. The identity \eqref{bga} is precisely Theorem 1 of \cite{duf96} with $a=1$ and $b=c=1/2$.
\end{proof}

Which Brownian path fragments yield a stationary sequence as constructed in Proposition \ref{stationary}?
More precisely, in the framework of Section \ref{seq}, we want to determine for which settings  
\ban
\rho_0/\sqrt{\tau_0}\ed\sqrt{2\Gamma_{3/2}U}. \label{stat}
\ee 
For example, a standard Brownian meander has $(\tau_0, \rho_0) \ed (1, \sqrt{2\Gamma_1})$,
so that $\rho_0/\sqrt{\tau_0}~=\sqrt{2\Gamma_1}$.  But the distribution of $\Gamma_1$ and $\Gamma_{3/2}U$ are not the same,
since their means are $1$ and $3/4$, respectively.
However, in the following two examples, we will recover natural stationary sequences.

First, consider the sequential construction of Section \ref{seq} in terms of Groeneboom's construction
\cite{gboom83} %%Groeneboom, Piet. The concave majorant of {B}rownian motion.Ann. Probab. (1983)construction of the \cmaj\ of a standard Brownian motion 
of the \cmaj\ of a standard Brownian motion $B$ on $(0, \infty)$
as embellished by Pitman \cite{p83} and \Cinlar\ \cite{brownistan}. Of course, the \cmaj\ of $B$ is
minus one times the \cmin\ of $-B$. 
Our notation largely follows \Cinlar.
Fix $a >0$, let
\[
%\lb{za}
Z(a) : = \max_{t \ge 0} \{ B(t) - a t \} = \inf\{ x : x + at > B(t) \mbox{ for all } t \ge 0 \}
\]
and let $D(a)$ denote the time at which the max is attained. So $(D(a), Z(a) + a D(a) )$ is one vertex of the \cmaj\ of $B$.
Let $S_{-1} < S_{-2} < \cdots $ denote the successive slopes of the \cmaj\ to the left of $D(a)$, so $ a < S_{-1}$ almost surely.

We can now spell out a sequential construction of the \cmaj\ of Brownian motion
starting at time $D(a)$ and working from right to left. This is similar in principle, but more complex in detail, to the description provided by
\Cinlar \cite[(3.11),(3.12),(3.13)]{brownistan}, which works from left to right, and the construction given in \cite{MR2007793}.  

\begin{corollary}
Define the vertex-intercept sequence $(\tau_j, \rho_j) =\left(D(S_{-j-1}), Z(S_{-j})\right)$ for $j\geq1$ and
\eq
\rho_0 = Z(D(a)) \mbox{ and } \tau_0 = D(a).
\en
Then for all $a>0$, the sequence $(\tau_j, \rho_j)_{j\geq0}$ satisfies the \tr\ recursion and the process $\left(\rho_j/\sqrt{\tau_j}\right)_{j\geq0}$ is stationary.
\end{corollary}
\begin{proof}
According to the Williams decomposition of $B$ at time $D(a)$,  there is the equality in distribution of conditioned processes
\eq
\lb{st1}
(B(v)- av, 0 \le v \le D(a) \giv Z(a) = r, D(a) = t)
\ed  (r - X(t-v), 0 \le v \le t)
\en
for $X$ a $BES(3)$ bridge from $(0,0)$ to $(t,r)$.
It now follows from \re{st1} and Corollary \ref{firstpassage} that the sequence of pairs $(\tau_i, \rho_i)$ follows the \tr\ recursion.

To show the claim of stationary, it is enough to show that 
\ban
Z(D(a))/\sqrt{D(a)} \ed \sqrt{ 2 \Gamma_{3/2} U}, \label{amt}
\ee
for $U$ a uniform $(0,1)$ random variable independent of the gamma variable.  
However, \eqref{amt} follows easily from 
\Cinlar\ [Remark 3.2]\cite{brownistan} which gives the representation for $a >0$
\ba
a^2 D(a) = 2 \Gamma_{3/2} (1 - \sqrt{U})^2; ~~~~ a Z(D(a)) = 2 \Gamma_{3/2} \sqrt{U} (1 - \sqrt{U}).
\ee
\end{proof}

Our second construction of a stationary sequence as indicated by Proposition \ref{stationary} is derived from a standard Brownian bridge.   
Recall that
$$
0 < \cdots < \alpha_{-2} < \alpha_{-1} < \alpha_0 < \alpha_1 < \alpha_2 < \cdots < 1
$$
with $\alpha_-n \downarrow 0$ and $\alpha_n \uparrow 1$ as $n \te \infty$ denote the times
of vertices of the convex minorant of a Brownian motion $B$ on $[0,1]$, arranged relative to
$$
\alpha_0 := \argmin_{ 0 \le t \le 1 } B_t .
$$
The same random set of vertex times $\{ \alpha_i, i \in \ints \}$ can be
indexed differently as 
$$
\{ \alpha_i, i \in \ints \} = \{ \alphabr_i, i \in \ints \} 
$$
where 
$$
\alphabr_0 := \argmin_{ 0 \le t \le 1 } { B_t -t B_1} = \alpha_{J}
$$
for an integer-valued random index $J$, and 
$$
\alphabr_i = \alpha_{J + i }.
$$
See \cite{jpfb09} for further discussion of this relationship between the \cmin\ of a Brownian motion and bridge.
The following representation of the $\alphabr_i$ can be derived from
Denisov's decomposition for the unconditioned Brownian motion: for $n = 0,1, 2, \ldots$ we have
\ba
\alphabr_{-n} & = \taubr_{n} \alphabr_0 \\
\alphabr_{n} & = 1 - \htaubr_n ( 1 - \alpha_0 ) \ed 1 - \alphabr_{-n} 
\ee
where
\ban
0 = 1 - \taubr_0 < 1 - \taubr_1  < \cdots \label{kj1}
\ee
and 
\ban
0 = 1 - \htaubr_0 < 1 - \htaubr_1  < \cdots \label{kj2}
\ee
are the times of vertices of the convex minorants of two identically distributed 
{\em Brownian pseudo-meanders} derived by Brownian scaling of portions of the
the path of $(B_t -t B_1,0 \le t \le 1)$  on 
$[0,\alphabr_0]$ (with time reversed) and $[\alphabr_0,1]$ respectively. Note that the sequences 
$(\taubr_{n})$ and $(\htaubr_{n})$ are identicially distributed, but they are not independent of
each other, and neither are they independent of $\alphabr_0$. While this complicates
analysis of the  sequence $(\alphabr_i, i \in \ints)$, the Brownian pseudo-meander is of
special interest for a number of reasons, including the following corollary.

\begin{corollary}
Let $0 = 1 - \taubr_0 < 1 - \taubr_1  < \cdots$ be the times of the vertices of the \cmin\ of a Brownian pseudo-meander as defined above,
and let $\rhobr_1>\rhobr_2>\cdots$ be the process of the intercepts at time one of the extension of the faces of the \cmin\ as illustrated by Figure \ref{22}.
If $\rhobr_0$ is the value of the pseudo-meander at time one, then
the sequence $(\taubr_j, \rhobr_j)_{j\geq0}$ satisfies the \tr\ recursion and the process $\left(\rhobr_j/\sqrt{\taubr_j}\right)_{j\geq0}$ is stationary.
\end{corollary}
\begin{proof}
Due to Denisov's decomposition and the representation of the Brownian Bridge as  $(B_t -t B_1,0 \le t \le 1)$ for $B$ a brownian motion,
the pseudo meander is absolutely continuous with respect to a standard BES$(3)$ process with density depending only on the final value.
Thus,
Theorem \ref{besbr} implies that  $(\taubr_j, \rhobr_j)_{j\geq0}$ satisfies the \tr\ recursion.

From this point, in order to show stationarity we must show that 
\ban
\rhobr_0=\rhobr_0/\sqrt{\taubr_0}\ed \sqrt{2 U \Gamma_{3/2}}. \label{idento}
\ee
Now, the variables $(\taubr_{n})$ and $(\htaubr_{n})$ as defined by \eqref{kj1} and \eqref{kj2} are distributed
like the corresponding $\alpha_i, \tau_i$ and $\htau_i$ of Corollary \ref{lighted} conditioned on the event that $B(1)=0$.
By using Denisov's decomposition to obtain a joint density for the minimum, time of the minimum, and final
value of a Brownian motion on $[0,1]$, some calculation leads to
\ba
\frac{\IP\left(\alphabr_0\in dt,\, \alphabr_0 B_1-B_{\alpha_0}\in dx \right)}{dx \, dt}=\sqrt{\frac{2}{\pi}}\,\frac{x^2}{t^{3/2}(1-t)^{3/2}}\exp\left(-\frac{x^2}{2t(1-t)}\right).
\ee
After noting 
\ba
\rhobr_0\ed \frac{\alphabr_0 B_1-B_{\alpha_0}}{\sqrt{\alphabr_0}},
\ee
a straightforward computation implies \eqref{idento} and hence also the corollary.
\end{proof}

\subsection{Central Limit Theorem}

As a final complement to our results pertaining to the \tr\ recursion, we obtain the following central limit theorem.
\begin{theorem}\label{cltt}
If a sequence $(\tau_j, \rho_j)_{j\geq0}$ satisfies the \tr\ recursion with arbitrary initial distribution, then 
\ban
\frac{\log(\tau_n)+2n}{2\sqrt{n}}\convd Z, \mbox{\,\,{ \rm as} } n\te\infty, \label{cltwp}
\ee
where $Z$ is a standard normal random variable.
\end{theorem}
In order to prove the theorem, we view $\tau_n$ as a function of a Markov chain and then
apply known results from ergodic theory.  We will need the following lemmas.
\begin{lemma}\label{eclt}(\cite{MR1287609} Theorem 17.4.4)
Suppose that $X_1, X_2, \ldots$ is a positive, Harris recurrent Markov chain with (nice) state space $\Omega$ and let $X$ be a random variable distributed as the
stationary distribution of the chain.  Suppose also that $g$ is a function on $\Omega$
and there is a function $\gh$ which satisfies
\ban
\gh(x)-\hp(x)=g(x)-\IE g(X), \label{gha}
\ee 
where
\ba
\hp(x):=\IE\left[\gh(X_2)|X_1=x\right].
\ee
If $\IE\gh(X)^2 < \infty$ and 
\ban
\sigma_g^2:=\IE\left[\gh(X)^2-\hp(X)^2\right] \label{sig}
\ee
is strictly positive, then
\ba
\frac{\sum_{i=1}^n g(X_i)- n\IE g(X)}{\sqrt{n}\sigma_g}\convd Z, \mbox{\,\,{ \rm as} } n\te\infty,
\ee
where $Z$ is a standard normal random variable.
\end{lemma}

\begin{lemma}\label{recu}
Let $(B_i)_{i\geq1}$ and $(C_i)_{i\geq1}$ be two i.i.d. sequences of positive random variables (not necessarily with equal distribution) such that
\ba
\IE\log(B_1)<0, \mbox{ and } \,\IE\log(C_1)<\infty.
\ee
If $X_0$ is a positive
random variable independent of $(B_i)_{i\geq1}$ and $(C_i)_{i\geq1}$, and for $n\geq0$, we define
\ba
X_{n+1}=B_{n+1} (X_n+C_{n+1}),
\ee
then there is a unique stationary distribution of the Markov chain 
$(X_n,C_{n+1})_{n\geq0}$.  
Moreover, if $(X,C)$ has this stationary distribution and
\ba
g(v,w):=\log\left(\frac{v}{v+w}\right),
\ee
then \eqref{gha} is satisfied for
\ba
\gh(v,w):=\log(v),
\ee
if and only if
\ba
\IE\log(B_1)=\IE g(X, C).
\ee
\end{lemma}
\begin{proof}
The existence and uniqueness of the stationary distribution can be easily read from the introduction of \cite{duf96}.  For the second assertion, note that
\ba
\gh(v,w)-\IE\left[\gh(X_1, C_2)|X_0=v, C_1=w\right]&=\log(v)-\IE\log(B_1)-\log(v+w) \\
	&=g(v,w)-\IE\log(B_1),
\ee
which proves the lemma.
\end{proof}

We can now prove our main result.
\begin{proof}[Proof of Theorem \ref{cltt}]
Let the \tr\ recursion be generated by the sequences
$(U_i)_{i\geq0}$ of i.i.d. uniform $(0,1)$ random variables and $(Z_i)_{i\geq1}$ of i.i.d. standard normal variables.  Note that we are
using the indexing of the \tr\ recursion as defined in the introduction. 

Next, we define
$Y_n:=U_n\rho_n/\sqrt{\tau_n}$ for $n\geq0$ so that 
\ban
Y_{n+1}^2=U_{n+1}^2\left(Y_{n}^2+Z_{n+1}^2\right), \label{nre}
\ee
and
\ban
\frac{\tau_{n+1}}{\tau_n}=\frac{Y_n^2}{Z_{n+1}^2+Y_n^2}. \label{oio}
\ee
We now have
\ba
\tau_n=\left(\frac{\tau_n}{\tau_{n-1}}\right)\left(\frac{\tau_{n-1}}{\tau_{n-2}}\right)\cdots\left(\frac{\tau_1}{\tau_0}\right)\tau_0,
\ee
which by applying \eqref{oio} yields
\ban
\log(\tau_n)-\log(\tau_0)&=\sum_{i=1}^n\log\left(\frac{Y_{i-1}^2}{Z_{i}^2+Y_{i-1}^2}\right). \label{gct}
\ee
%and by appealing to \eqref{nre} yields
%\ban
%\log(\tau_n)=2 \log(Y_0)+2\sum_{i=1}^{n-1}\log(U_i)-\log(Z_n^2+Y_{n-1}^2). \label{gcp}
%\ee
%
%
%The representation \eqref{gct} allows us to apply the ergodic machinery of Lemma \ref{eclt}, while the expression \eqref{gcp} will prove useful
%for direct calculations.  More precisely, 
We  have the following framework:
\ban
\log(\tau_n)-\log(\tau_0)=\sum_{i=1}^ng(Y_{i-1}^2, Z_i^2), \label{tyu}
\ee 
where  
\ban
g(v,w):=\log\left(\frac{v}{v+w}\right) \label{g}
\ee
and $(Y_n^2, Z_{n+1}^2)_{n\geq0}$ is a Markov chain on $\IR^+\times\IR^+$ given by \eqref{nre} and where the 
distribution of $Y_0$ is arbitrary.

By Lemma \ref{recu}, we can apply Lemma \ref{eclt} with $\gh(v,w)=\log(v)$ to \eqref{tyu} as long as
\ban
\IE\log(U_1^2)=\IE\log\left(\frac{Y^2}{Y^2+Z^2}\right), \label{idif}
\ee
where $(Y^2,Z^2)$ are distributed as the stationary distribution of the chain given by \eqref{nre}.  This stationary distribution is unique by Lemma \ref{recu}, and
it is straightforward to see that $Z$ is standard normal, independent of $Y$, and $Y^2\ed2U\Gamma_{1/2}$, where $U$ is uniform $(0,1)$ independent of $\Gamma_{1/2}$.   From this point, it is easy to see that \eqref{idif} is equivalent to
\ba
\IE\log(U)=\IE\log\Gamma_{1/2}-\IE\log(U\Gamma_{1/2}+\Gamma_{1/2}'),
\ee
where all variables appearing are independent.  Some calculations show
$\IE\log(U)=-1$ and $\IE\log(\Gamma_{1/2})=-2\log(2)-\gamma$, where $\gamma$ is Euler's constant.  Also, since $U\ed\Gamma_1/(\Gamma_1+\Gamma_1')$, 
Theorem 1 of \cite{duf96} implies that 
\ban
U\Gamma_{1/2}+\Gamma_{1/2}'\ed U\Gamma_{3/2}, \label{cin}
\ee
so that \eqref{idif} follows after noting $\IE\log(\Gamma_{3/2})=2-\gamma-2\log(2)$.
We remark in passing that \eqref{idif} implies $\IE g(Y^2,Z^2)=-2$, which is the desired mean constant in applying Lemma \ref{eclt} to obtain the
expression \eqref{cltwp}.

Applying Lemma \ref{eclt} with $\gh(v,w)=\log(v)$, the theorem will be proved for \eqref{tyu} if we can show 
\ba
\IE\left[\log^2(Y^2)\right]<\infty,
\ee
which is straightforward, 
and
\ban
\IE\left[\log^2(Y^2)-\left(-2+\log(Y^2+Z^2)\right)^2\right]=4. \label{pst}
\ee
Using \eqref{cin}, some algebra reveals that \eqref{pst} is equivalent to
\ba
\IE\log(\Gamma_{1/2})^2&+2\IE\left[(\log(2)+\log(U))\log(\Gamma_{1/2})\right]+4\log(2)+4\log(U) \\
&=\IE\log(\Gamma_{3/2})^2+2\IE\left[(\log(2)+\log(U)-2)\log(\Gamma_{3/2})\right]+8,
\ee
where the random variables are the same as above.
This equality is easily verified using the moment information above and the facts
\ba
\IE\log(\Gamma_{1/2})^2=\frac{\pi^2}{2}+(\gamma+2\log(2))^2
\ee
and
\ba
\IE\log(\Gamma_{3/2})^2=\frac{\pi^2}{2}+(\gamma+2\log(2)-2)^2-4.
\ee

Finally, we have shown the CLT for \eqref{tyu}, and  
\eqref{cltwp} follows since $\log(\tau_0)/\sqrt{n}\rightarrow0$ in probability.
\end{proof}

\section{Appendix}

This appendix provides the calculations involved in obtaining information about the times
of vertices of the \cmin\ of Brownian motion on $[0, 1]$ from analogous facts about the times of vertices of the \cmin\ of the
standard meander; see Corollary \ref{alpim}.  

Following the previous notation, let 
$(\alpha_i)_{i\in\IZ}$ be the times of the vertices of the \cmin\ of a Brownian motion on $[0,1]$ as described in the introduction by \eqref{alps} and 
\eqref{alp0}, and let $f_{\alpha_i}$ denote the density of $\alpha_i$.  
Also, for $n=1, 2, \ldots$ let $1-\tau_n$ be the time of the right endpoint of the $n$th face of the \cmin\ of a standard meander, and let 
$f_{\tau_n}$ denote the density of $\tau_n$.  As per Corollary \ref{lighted}, we have for $n\geq0$ the representation 
\ba
\alpha_{-n}=\alpha_0\tau_n, 
\ee
where $\alpha_0$ is arcsine distributed and independent of $\tau_n$.  

For example, for each $n = 1, 2, \ldots$ we can compute directly
\ban
f_{\alpha_{-n}}(u) = \frac{1} {\pi} \int_{u}^ 1  v^{-3/2} (1-v)^{-1/2} f_{\tau_n}(u/v) dv \label{dtu}
\ee
and for $p > 0$
\ban
\IE ( \alpha_{-n} ^ p ) = \E ( \alpha_{0} ^ p ) \E( \tau_n ^p ), \label{mtu}
\ee
and
%\ba
%\E ( \alpha_{0} ^ p ) = 
%\frac{ \beta(\hf, \hf + p) } { \beta(\hf, \hf ) } = \frac{ \Gamma(\hf + p)   } { \Gamma ( 1 + p ) \Gamma(\hf ) }
%\ee
expressions for $ \E ( \alpha_{0} ^ p )$ are known.  Equations \eqref{dtu} and \eqref{mtu} can be used to transfer moment and
density information
from $\tau_n$ to $\alpha_{-n}$, and also note that $\alpha_{n}\ed1-\alpha_{-n}$, so that this program 
yields the analogous properties for $\alpha_n$. 
Unfortunately, \eqref{dtu} can be difficult to handle, so we use the following proposition.
\begin{proposition}\label{asiv}
Let $(c_n)_{n\geq0}$ be a sequence of non-negative numbers such that
\ba
\sum_{n=0}^\infty c_n(1-u)^n 
\ee
converges for all $0<u\leq1$.  If
\ba
g(u):=(1-u)^{-a}\sum_{n=0}^\infty c_n(1-u)^n 
\ee
for some $0\leq a<1$, and
\ba
f(u):=\frac{1}{\sqrt{\pi u}}\sum_{n=0}^\infty \frac{\Gamma(n-a+1)}{\Gamma(n-a+\frac{3}{2})}c_n (1-u)^{n-a+\hf}
\ee
then 
\ban
f(u) = \frac{1} {\pi} \int_{u}^ 1  v^{-3/2} (1-v)^{-1/2}g(u/v) dv. \label{kid}
\ee
\end{proposition} 
\begin{proof}
The proposition follows from term by term integration using the fact that for $p>0$,
\ba
\frac{\Gamma(p+1)}{\Gamma(\frac{1}{2})\Gamma(p+\frac{1}{2})}u^{-1/2}(1-u)^{p-\hf } = \frac{1} {\pi} \int_{u}^ 1  v^{-3/2} (1-v)^{-1/2}\left[p\left(1-\frac{u}{v}\right)^{p-1}\right] dv,
\ee
which is derived by considering densities in the standard identity
\ba
\beta_{1/2,1/2} \beta_{1,p}  \ed \beta_{1/2,p+1/2}
\ee
where $\beta_{b,d}$ denotes a random variable with beta$(b,d)$ distribution for some $b,d>0$, and on the left side the
random variables $\beta_{1/2,1/2}$ and $\beta_{1,p}$ are independent.
\end{proof}

In order to illustrate the method, we will use Proposition \ref{asiv} to finish the proof of Corollary \ref{alpim}.
In order to ease exposition, we will refer to $f$ of \eqref{kid} as the \emph{arcsine transform} of $g$.
Now, recall that
\ba
\sum_{i=1}^{\infty}f_{\alpha_{-i}}(t)=\frac{1}{\pi} \int_t^1v^{-3/2}(1-v)^{-1/2} \left(\sum_{i=1}^{\infty}f_{\tau_i}(t/v)\right) dv,
\ee
and that 
\ba
\sum_{i=1}^\infty f_{\tau_i}(u)=\frac{1}{4}\left[\frac{1}{u\sqrt{1-u}}+\frac{1}{u}+\left(\frac{1}{1-u}-\frac{\arcosh(u^{-1/2})}{(1-u)^{-3/2}}\right)\right].
\ee
We claim that 
\ban
\sum_{i=1}^{\infty}f_{\alpha_{-i}}(u)=\frac{1}{4}\left[\frac{1}{u}+\frac{2}{\pi u}\arcos(\sqrt{u})+\frac{2}{\pi}\left(\frac{ \arcos \sqrt{ u } }{ 1 - u } - \frac{ 1 } {\sqrt{u} \sqrt{ 1 - u }} \right)\right], \label{www}
\ee
which will follow by applying Proposition \ref{asiv} appropriately.  More precisely, we can write
\ba
u^{-1}=\sum_{n=0}^\infty(1-u)^n,
\ee
so that Proposition \ref{asiv} with $a=0$ and $c_n\equiv1$ implies the arcsine transform of $u^{-1}$ can be represented as
\ban
\frac{2}{\pi}\sqrt{\frac{1-u}{u}}\sum_{n=0}^\infty \frac{n!}{(\thf)_n}(1-u)^n &= \frac{2}{\pi}\sqrt{\frac{1-u}{u}}\,_2F_1(1,1;\thf;(1-u)) \notag \\
		%&= \frac{ \arcsin \sqrt{1-u} }{ \sqrt{u (1-u)}  } \notag \\
		&= \frac{2}{\pi u}\arcos(\sqrt{u}), \label{fp1}
\ee
where $(a)_n=a(a+1)\cdots(a+n-1)$ and in the second inequality we have used the evaluation of $_2F_1$ found in (15.1.6) of \cite{as}.
%and the third follows from the fact
%and $\arcsin \sqrt{1-u}  = \arcos \sqrt{u}$.

Similarly, we can apply Proposition \ref{asiv} with $a=1/2$ and $c_n\equiv1$ to find the arcsine transform of $[u\sqrt{(1-u)}]^{-1}$ to be
\ban
u^{-\hf}  \sum_{n = 0}^\infty \frac{ \Gamma(n + \hf)  }{\Gamma(\hf) n! } (1 - u)^n &=  u^{-1}. \label{fp2}
\ee

Finally, we write 
\ban
\frac{\sqrt{1-u}-\arcosh(u^{-1/2})}{(1-u)^{-3/2}}= -\sum_{n=0}^\infty \frac{ (1-u)^n } { 2 n + 3 },  \label{fs3}
\ee
so that we can apply Proposition \ref{asiv} with $a=0$ and $c_n=1/(2n+3)$ to find the arcsine transform of \eqref{fs3} to be
\ban
\frac{2}{\pi}\sqrt{\frac{1-u}{u}}\sum_{n=0}^\infty &\frac{ n!  (1-u)^n } { (\thf)_n  (2 n + 3 ) } = \frac{2}{3\pi}\sqrt{\frac{1-u}{u}}\sum_{n=0}^\infty \frac{ n!  (1-u)^n } { (\fhf)_n } \notag \\
&= \frac{2}{3\pi}\sqrt{\frac{1-u}{u}} \,_2F_1( 1,1;\fhf; 1-u ) \notag \\
&= \frac{2}{\pi}\sqrt{\frac{1-u}{u}}\left((1-u)^{-1} -(1-u)^{-3/2} u^{1/2} \arcos(\sqrt{u})\right), \label{fp3}
\ee
where in the last equality we have used the reduction formula (15.2.20) of \cite{as}, and then again (15.1.6) there.

Now combining \eqref{fp1}, \eqref{fp2}, and \eqref{fp3} shows \eqref{www} and proves Corollary \ref{alpim}.  As mentioned previously,
Proposition \ref{asiv} can also be used to obtain expressions for $f_{\alpha_i}$ by expanding $f_{\tau_i}$ appropriately.

\bibliographystyle{plain}
\bibliography{convex_minorant}

\begin{thebibliography}{10}

\bibitem{as}
M.~Abramowitz and I.~Stegun, editors.
\newblock {\em Handbook of mathematical functions with formulas, graphs, and
  mathematical tables}.
\newblock Dover Publications Inc., New York, 1992.
\newblock Reprint of the 1972 edition.

\bibitem{ap10}
J.~Abramson and J.~Pitman.
\newblock Concave majorants of random walks and related {P}oisson processes.
\newblock Preprint, 2010.

\bibitem{jpfb09}
F.~Balabdaoui and J.~Pitman.
\newblock The distribution of the maximal difference between brownian bridge
  and its concave majorant, 2009.
\newblock Preprint, arXiv:0910.0405.

\bibitem{MR770946}
R.~F. Bass.
\newblock Markov processes and convex minorants.
\newblock In {\em Seminar on probability, {XVIII}}, volume 1059 of {\em Lecture
  Notes in Math.}, pages 29--41. Springer, Berlin, 1984.

\bibitem{MR1406564}
J.~Bertoin.
\newblock {\em L\'evy processes}, volume 121 of {\em Cambridge Tracts in
  Mathematics}.
\newblock Cambridge University Press, Cambridge, 1996.

\bibitem{MR1747095}
J.~Bertoin.
\newblock The convex minorant of the {C}auchy process.
\newblock {\em Electron. Comm. Probab.}, 5:51--55 (electronic), 2000.

\bibitem{bcp03}
J.~Bertoin, L.~Chaumont, and J.~Pitman.
\newblock Path transformations of first passage bridges.
\newblock {\em Electron. Comm. Probab.}, 8:155--166 (electronic), 2003.

\bibitem{blum_exc83}
R.~M. Blumenthal.
\newblock Weak convergence to {B}rownian excursion.
\newblock {\em Ann. Probab.}, 11(3):798--800, 1983.

\bibitem{MR2007793}
C.~Carolan and R.~Dykstra.
\newblock Characterization of the least concave majorant of {B}rownian motion,
  conditional on a vertex point, with application to construction.
\newblock {\em Ann. Inst. Statist. Math.}, 55(3):487--497, 2003.

\bibitem{chaumont-bravo09}
L.~Chaumont and G.~Uribe~Bravo.
\newblock Markovian bridges: weak continuity and pathwise constructions, 2009.
\newblock Preprint, arXiv:0905.2155.

\bibitem{brownistan}
E.~{\c{C}}inlar.
\newblock Sunset over {B}rownistan.
\newblock {\em Stochastic Process. Appl.}, 40(1):45--53, 1992.

\bibitem{MR726906}
I.~V. Denisov.
\newblock Random walk and the {W}iener process considered from a maximum point.
\newblock {\em Teor. Veroyatnost. i Primenen.}, 28(4):785--788, 1983.

\bibitem{duf96}
D.~Dufresne.
\newblock On the stochastic equation {${\mathcal{L}}(X)={\mathcal{L}}[B(X+C)]$}
  and a property of gamma distributions.
\newblock {\em Bernoulli}, 2(3):287--291, 1996.

\bibitem{fpy92}
P.~Fitzsimmons, J.~Pitman, and M.~Yor.
\newblock Markovian bridges: construction, palm interpretation, and splicing.
\newblock In E.~{\c{C}}inlar, K.L. Chung, and M.J. Sharpe, editors, {\em
  Seminar on Stochastic Processes, 1992}, pages 101--134. Birkh{\"a}user,
  Boston, 1993.

\bibitem{MR896736}
P.~J. Fitzsimmons.
\newblock Another look at {W}illiams' decomposition theorem.
\newblock In {\em Seminar on stochastic processes, 1985 ({G}ainesville, {F}la.,
  1985)}, volume~12 of {\em Progr. Probab. Statist.}, pages 79--85.
  Birkh\"auser Boston, Boston, MA, 1986.

\bibitem{Freedman1983}
D.~Freedman.
\newblock {\em Brownian motion and diffusion}.
\newblock Holden-Day, San Francisco, Calif., 1971.

\bibitem{MR994088}
C.~M. Goldie.
\newblock Records, permutations and greatest convex minorants.
\newblock {\em Math. Proc. Cambridge Philos. Soc.}, 106(1):169--177, 1989.

\bibitem{gp80}
P.~Greenwood and J.~Pitman.
\newblock Fluctuation identities for {L}{\'e}vy processes and splitting at the
  maximum.
\newblock {\em Advances in Applied Probability}, 12:893--902, 1980.

\bibitem{gboom83}
P.~Groeneboom.
\newblock The concave majorant of {B}rownian motion.
\newblock {\em Ann. Probab.}, 11(4):1016--1027, 1983.

\bibitem{imhof84}
J.-P. Imhof.
\newblock Density factorizations for {B}rownian motion, meander and the
  three-dimensional {B}essel process, and applications.
\newblock {\em J. Appl. Probab.}, 21(3):500--510, 1984.

\bibitem{MR942038}
J.-F. Le~Gall.
\newblock Une approche \'el\'ementaire des th\'eor\`emes de d\'ecomposition de
  {W}illiams.
\newblock In {\em S\'eminaire de {P}robabilit\'es, {XX}, 1984/85}, volume 1204
  of {\em Lecture Notes in Math.}, pages 447--464. Springer, Berlin, 1986.

\bibitem{MR1287609}
S.~P. Meyn and R.~L. Tweedie.
\newblock {\em Markov chains and stochastic stability}.
\newblock Communications and Control Engineering Series. Springer-Verlag London
  Ltd., London, 1993.

\bibitem{MR1739699}
M.~Nagasawa.
\newblock {\em Stochastic processes in quantum physics}, volume~94 of {\em
  Monographs in Mathematics}.
\newblock Birkh\"auser Verlag, Basel, 2000.

\bibitem{p75}
J.~Pitman.
\newblock {One-dimensional Brownian motion and the three-dimensional Bessel
  process}.
\newblock {\em Advances in Applied Probability}, 7:511--526, 1975.

\bibitem{p83}
J.~Pitman.
\newblock Remarks on the convex minorant of {B}rownian motion.
\newblock In {\em Seminar on Stochastic Processes, 1982}, pages 219--227.
  Birkh{\"a}user, Boston, 1983.

\bibitem{MR2245368}
J.~Pitman.
\newblock {\em Combinatorial stochastic processes}, volume 1875 of {\em Lecture
  Notes in Mathematics}.
\newblock Springer-Verlag, Berlin, 2006.
\newblock Lectures from the 32nd Summer School on Probability Theory held in
  Saint-Flour, July 7--24, 2002, With a foreword by Jean Picard.

\bibitem{pub10}
J.~Pitman and G.~Uribe-Bravo.
\newblock The convex minorant of a {L}{\'e}vy process.
\newblock Preprint, 2010.

\bibitem{suidan01}
T.~M. Suidan.
\newblock Convex minorants of random walks and {B}rownian motion.
\newblock {\em Teor. Veroyatnost. i Primenen.}, 46(3):498--512, 2001.

\bibitem{MR0350881}
D.~Williams.
\newblock Path decomposition and continuity of local time for one-dimensional
  diffusions. {I}.
\newblock {\em Proc. London Math. Soc. (3)}, 28:738--768, 1974.

\end{thebibliography}
%\bibliography{convex_minorant,/server/mirror/data/pub/users/pitman/bibserver}
%\bibliography{MR,http://www.stat.berkeley.edu/~pitman/convex_minorant,http://www.stat.berkeley.edu/~pitman/bibserver}
\end{document}